\documentclass[12pt,a4paper]{amsart}
\theoremstyle{plain}
\usepackage{enumerate,amsfonts,amssymb,mathrsfs,amscd}
\usepackage{amsmath}
\usepackage[english]{babel}
\usepackage[utf8]{inputenc}
\usepackage{amsmath, amsthm, amssymb}
\usepackage[pdftex]{graphicx,color}

\usepackage{tikz,pgf}
\usetikzlibrary{calc}
\usepackage[colorlinks=true,linkcolor=red,citecolor=red]{hyperref}

\usepackage{etex}

\usepackage[shortlabels]{enumitem}
\usepackage{setspace}
\usepackage{tabularx,multirow}
\usepackage[arrow, matrix, curve]{xy}
\usepackage{MnSymbol}

\advance\hoffset-20mm \advance\textwidth40mm

\usepackage[normalem]{ulem}

\theoremstyle{plain}

\newtheorem{thm}{Theorem}[section]
\newtheorem{cor}[thm]{Corollary}
\newtheorem{lem}[thm]{Lemma}
\newtheorem{prop}[thm]{Proposition}
\theoremstyle{definition}
\newtheorem{defi}[thm]{Definition}

\newtheorem{conj}[thm]{Conjecture}

\newtheorem{nota}[thm]{Notation}
\newtheorem{rem}[thm]{Remark}
\newtheorem{rems}[thm]{Remarks}
\newtheorem{exa}[thm]{Example}

\newtheorem{sit}[thm]{}

\newcounter{caseinproof}
0

\newcommand{\Der}{\operatorname{Der}}
\def\AA{{\mathbb A}}

\def\NN{{\mathbb N}}
\def\ZZ{{\mathbb Z}}
\def\QQ{{\mathbb Q}}
\def\TT{{\mathbb T}}
\newcommand{\T}{{\mathbb T}}
\newcommand{\A}{{\mathbb A}}
\newcommand{\N}{{\mathbb N}}
\newcommand{\Z}{{\mathbb Z}}
\def\kk{{\mathbb{K}}}
\def\GG{{\mathbb{G}}}
\def\CC{{\mathbb C}}

\def\cO{{\mathcal{O}}}

\newcommand{\p}{{\partial}}

\def\span{\mathop{\span}}
\def\Aut{\mathop{\rm Aut}}
\def\Aff{\mathop{\rm Aff}}
\def\SAut{\mathop{\rm SAut}}
\def\SAff{\mathop{\rm SAff}}
\def\ad{\mathop{\rm ad}}
\def\Ad{\mathop{\rm Ad}}
\def\LND{\mathop{\rm LND}}
\def\SL{\mathop{\rm SL}}
\def\sl{\mathop{\rm sl}}

\def\Cl{\mathop{\rm Cl}}

\def\Stab{\mathop{\rm Stab}}
\def\Spec{\mathop{\rm Spec}}
\def\Transl{\mathop{\rm Transl}}
\def\Tame{\mathop{\rm Tame}}
\def\STame{\mathop{\rm STame}}

\def\Der{\mathop{\rm Der}}

\def\deg{\mathop{\rm deg}}

\def\exp{\mathop{\rm exp}}

\def\reg{\mathop{\rm reg}}

\def\Hom{\mathop{\rm Hom}}
\def\Cox{\mathop{\rm Cox}}
\def\Jac{\mathop{\rm Jac}}

\def\span{\mathop{\rm span}}

\def\ll1{l_{\lambda}^{-1}(1)}
\def\lm1{l_{\mu}^{-1}(1)}

\DeclareRobustCommand{\indlim}{\varinjlim\nolimits}

\newcommand\qmatrix[2][1]{\left(\renewcommand\arraystretch{#1}
\begin{equation}gin{array}{*{20}r}#2\end{array}\right)}

\thanks{
The first author was supported by the grant RSF 19-11-00172. The second author was partially 
supported by the HSE University Basic Research Program, Russian Academic 
Excellence Project '5-100'.}
\thanks{This work
was done during a stay of the first and the second authors at the  Institut
Fourier, Grenoble, France. They would like to thank this
institution for  hospitality, support, and excellent working conditions.}

\begin{document}

\title{Infinite transitivity, finite generation,  and Demazure roots}
\author{I.~Arzhantsev, K.~Kuyumzhiyan, and M.~Zaidenberg}
\address{National Research University Higher School of 
Economics, Russian Federation, Faculty of Computer Science, 3
Kochnovskiy Proezd, Moscow, 125319 Russia}
\email{arjantsev@hse.ru}
\address{National Research University Higher School of 
Economics, Russian Federation,  Faculty of
Mathematics, 6\ Usacheva str.,
Moscow, 119048 Russia}
\email{karina@mccme.ru}
\address{Universit\'e Grenoble Alpes, CNRS, Institut Fourier, F-38000 Grenoble, France
}
\email{Mikhail.Zaidenberg@univ-grenoble-alpes.fr}
\date{}

\begin{abstract} 
An affine algebraic variety $X$ of dimension $\ge 2$ is called \emph{flexible} if the subgroup $\SAut(X)\subset\Aut(X)$ generated by the one-parameter unipotent subgroups acts $m$-transitively on ${\rm reg}\,(X)$ for any $m\ge 1$. In \cite{AKZ} we proved that any nondegenerate toric affine variety $X$ is flexible. In the present paper we show that one can find a subgroup of $\SAut(X)$ generated by a finite number of one-parameter unipotent subgroups which has the same transitivity property, provided the  toric variety $X$ is smooth in codimension two.  For $X=\mathbb{A}^n$ with $n\ge 2$,  three such subgroups suffice. 
\end{abstract} 
\maketitle

\thanks{
{\renewcommand{\thefootnote}{} \footnotetext{ 2010
\textit{Mathematics Subject Classification:}
14R20,\,32M17.\mbox{\hspace{11pt}}}}

{\footnotesize \tableofcontents}

\section{Introduction} Let  $\kk$ be  an algebraically closed field of  characteristic zero, and
let $X$ be an affine variety  over $\kk$. A one-parameter subgroup $H$ of $\Aut(X)$ isomorphic as an algebraic group to the additive group $\GG_a$ of the base  field $\kk$ is called a \emph{unipotent} one-parameter subgroup, or a \emph{$\GG_a$-subgroup}, for short. One can consider the subgroup $\SAut(X)\subset\Aut(X)$ generated by all the one-parameter unipotent subgroups. It is known (\cite[Thm.\ 2.1]{AKZ}) that for a toric affine variety $X$ with no torus factor (that is, a nondegenerate toric variety) the group  $\SAut(X)$ acts infinitely transitively on the smooth locus $\reg (X)$, that is, $m$-transitively for any $m\ge 1$. Varieties $X$ with this property are called \emph{flexible} (\cite{AFKKZ, AKZ}). Actually, the simple transitivity of $\SAut(X)$ in $\reg(X)$ already guarantees that $X$ is flexible (\cite[Thm.\ 0.1]{AFKKZ}). The same is true for quasi-affine varieties (\cite{APS, FKZ}). In turn, the flexibility implies several other useful properties, for instance, the unirationality, see, e.g., \cite{AFKKZ, BKK, Pop14}. The flexibility has found important applications, e.g., to the Zariski cancellation problem (\cite{FKZ-I}). It is known (\cite[Thm.\ 1.1]{FKZ}) that the flexibility survives upon passing to the complement of a subvariety of codimension at least 2. Different flexibility properties are intensively studied in complex analytic geometry, see, e.g.,  \cite{AFKKZ,  For, KK}.

Besides the  toric affine  varieties  with no torus factor
 there are several other interesting classes of flexible affine varieties,  see, e.g., \cite{AFKKZ, AKZ,  APS, BEE, Dub, MPS, PW, Per,  Pop14, PZ, Sha}.   

In fact, for a flexible $X$ certain proper subgroups $G$ of $\SAut(X)$ act also  infinitely transitively on $\reg(X)$, or at least on a Zariski dense open subset of $\reg(X)$. This is the case for a subgroup $G$  generated by a sufficiently rich family of $\GG_a$-subgroups of  $\SAut(X)$, see \cite[Thm.\ 2.2]{AFKKZ}. Let us stay on this in more detail. 

Let $\LND(X)$ stand for the set of all nonzero locally nilpotent derivations (LNDs, for short) of the structure algebra $\cO_X(X)$. Letting $A=\cO_X(X)$ we also write $\LND(A)$ instead of $\LND(X)$. Any $\partial\in\LND(X)$ is a generator of  a $\GG_a$-subgroup $H=\exp(\kk\partial)$ of $\Aut(X)$, and any $\GG_a$-subgroup $H\subset\Aut(X)$ has the form $H=\exp(\kk\partial)$ for some $\partial\in {\LND}\,(X)$. If $a\in\ker\partial$ then $a\partial$ is again an LND called a \emph{replica} of $\partial$. The subgroup $H(a)=\exp(\kk a\partial)$ is called a \emph{replica} of $H$. 

A family $\mathfrak{F}$ of $\GG_a$-subgroups of $\Aut(X)$ is called \emph{saturated} if 
\\
(i) any replica of $H\in\mathfrak{F}$ belongs to $\mathfrak{F}$, and\\
(ii) $\mathfrak{F}$ is closed under conjugation by the elements of the subgroup $G=G(\mathfrak{F})$  generated by the members of $\mathfrak{F}$ (\cite[Def.\ 2.1]{AFKKZ}).\\
Notice (\cite[Lem.\ 4.6]{FKZ-I}) that for any family $\mathfrak{F}$ which verifies (i) there exists a larger family $\mathfrak{F}'$ satisfying both (i) and (ii) such that $G(\mathfrak{F})=G(\mathfrak{F}')$.
If $G(\mathfrak{F})$ has an open orbit in $X$ and $\mathfrak{F}$ is saturated then the action of $G$ on this orbit  is infinitely transitive (\cite[Thm.\ 2.2]{AFKKZ}). By  \cite[Prop.\ 2.15]{FKZ}  one can find such a family $\mathfrak{F}$ composed by the replicas of just two LNDs.

We say that $X$ is \emph{generically flexible} if $\SAut(X)$  acts on $X$ with an open orbit. For instance (\cite{Giz}), any Gizatullin surface $X$ is generically flexible. Notice that the open orbit of $\Aut(X)$
can be smaller than $\reg(X)$, see, e.g., \cite{Kov} for examples of Gizatullin surfaces with this property. 

The  assumption of saturation could be too restrictive in applications. The aim of the present paper is to elaborate more moderate conditions on the family $\mathfrak{F}$  which still guarantee the infinite transitivity of $G(\mathfrak{F})$ on the open orbit. In Section~\ref{sec:inf-trans} we formulate such conditions, see Theorem~\ref{th-inf-tr}. It occurs that for a generically flexible variety, there exists a countable family of $\GG_a$-subgroups which generates a group acting infinitely transitively on its open orbit, see Corollary~\ref{cor:countable}. We observe in Section~\ref{sec:orbits} that the orbits and the transitivity of an algebraically generated group $G\subset\Aut(X)$ are not affected upon passing to the closure $\overline{G}$, see Proposition~\ref{lem:many-points}. 
The remaining part of the paper appeared as a result of our discussions on the following

\begin{conj}\label{sit:conjecture} \emph{Any generically flexible affine variety $X$ admits a finite collection $\{H_1,\ldots,H_N\}$ of $\GG_a$-subgroups of $\Aut(X)$ such that the group $G=\langle H_1,\ldots,H_N\rangle $ acts infinitely transitively on its open orbit.}
\end{conj}

In Section~\ref{sec:toric-inf-trans} we fix this conjecture for toric affine varieties under a certain mild restriction.
In Section~\ref{sec:toric} we recall some basics on toric varieties and Cox rings. We deal there with the LNDs of the structure ring of a toric affine variety $X$ which are normalized by the  given torus $\TT$ acting on $X$ with an open orbit. In the sequel we refer to $\TT$ as to the \emph{acting torus} of $X$. The degree $e$ of such an LND $\p$ is a lattice vector called a \emph{Demazure root}. The corresponding $\GG_a$-subgroup $H_e\subset\Aut(X)$ normalized by $\TT$ is called a (\emph{Demazure}) \emph{root subgroup}.  

Our approach exploits hardly the following phenomenon. Consider a group $G$ acting effectively on a toric affine variety $X$ and generated by its unipotent subgroups. The closure $\overline{G}$ of $G$ with respect to the ind-topology could contain more Demazure root subgroups than the group $G$ itself. 
However, the multiple transitivity of $\overline{G}$  on its orbit is inherited by the group $G$, see Proposition~\ref{lem:many-points}. 

To describe some extra Demazure root subgroups contained in $\overline{G}$ we develop in Section~\ref{ss:degeneration} certain degeneration techniques. Given an LND $\p\in\Der(\cO_X(X))$ generating a one-parameter unipotent subgroup of $G$ we define its Newton polytope $N(\p)$ with respect to the acting torus. The extremal points of this polytope correspond to one-parameter unipotent root subgroups which belong to $\overline{G}$, see Proposition~\ref{prop: principal-part}. To find a convenient (non-root) LND $\p$ we conjugate one Demazure root subgroup by a second one which does not centralize the first. The Newton polytope $N(\p)$ of the resulting LND $\p$ occurs to be a segment with one of its endpoints being a desired extra Demazure root. This segment can be found explicitly by using a version of the Baker-Campbell-Hausdorff formula, see Corollary~\ref{cor:NP-Ad}. 

The simplest toric affine varieties are the affine spaces $\A^n=\A^n_\kk$. In this case, both the affine group ${\Aff}_n$ and the group SL$(n,\kk)$ extended by just one root subgroup act infinitely transitively on their open orbits, see, e.g., Corollary ~\ref{cor:SLn}. This is based on the results of Bodnarchuk (\cite{Bod02-2, Bod05}), Edo (\cite{Edo}), and Furter (\cite{Fur15}) concerning cotame automorphisms of the affine spaces, see Definition~\ref{sit:tame}. 
We establish the following facts, see Theorems~\ref{thm:V} and~\ref{thm:III}.

\begin{thm}\label{thm:main-1}
For any $n\ge 2$ one can find three $\GG_a$-subgroups of $\Aut(\A^n)$ which generate a subgroup  acting  infinitely transitively on $\A^n$. The same is true for some $n+2$ root subgroups. 
\end{thm}

Our main result for toric affine varieties  (see Theorem~\ref{thm:toric-inf-trans}) is the following

\begin{thm}\label{thm:main-2}
For any toric affine variety $X$ of dimension at least 2, with no torus factor, and smooth in codimention $2$, one can find 
a finite collection of Demazure root subgroups such that the group generated by these acts infinitely transitively on the smooth locus $\reg(X)$. \end{thm}

\section{Infinite transitivity on the open orbit}\label{sec:inf-trans}

We are working over an algebraically closed field $\kk$ of characteristic zero. We let $\AA^n$ stand for the affine space of dimension $n$ over $\kk$, and $\mathbb{G}_a$ and $\mathbb{G}_m$ for the additive and the multiplicative groups of $\kk$, respectively, viewed as algebraic groups.

\begin{sit}\label{sit-2.1} Let $X$ be an affine variety  over $\kk$ of dimension~$n\ge 2$. Consider  a finite collection of pairwise non-collinear locally nilpotent derivations $\partial_1,\partial_2,\ldots,\partial_k$ of $\cO_X(X)$ which contains a subset of 
$n$ linearly independent derivations. For every $i=1,\ldots,k$ fix a finitely generated subalgebra  $A_i\subset\ker\partial_i$ such that the fraction field  ${\rm Frac}\,(A_i)$ has finite index in ${\rm Frac}\,(\ker\partial_i)$. 
Consider the following possibilities:
\begin{itemize} \item[$(\alpha)$] $\cO_X(X)$ is generated by $A_1,\ldots,A_k$;
\item[$(\beta)$] $[{\rm Frac}\,(\ker\partial_i):{\rm Frac}\,(A_i)]=1$ for some value of $i$;
\item[$(\gamma)$] $[{\rm Frac}\,(\ker\partial_i):{\rm Frac}\,(A_i)]>1$ for all $i=1,\ldots,k$. In the latter case we fix an extra element $b_1\in \ker\partial_1$ such that ${\rm Frac}\,(\ker\partial_1)$ is generated by $b_1$ and ${\rm Frac}\,(A_1)$. In cases ($\alpha$) and ($\beta$) one might take $b_1=0$.
\end{itemize}
Let~$G$ be the subgroup of $\SAut(X)$ generated by the $\GG_a$-subgroups 
$$H_0=\exp(\kk b_1\partial_1)\,\,\,\mbox{and}\,\,\, H_i(a_i)=\exp(\kk a_i\partial_i)\quad\mbox{where}\quad a_i\in A_i,\,\,\, i=1,\ldots,k\,.$$ 
Notice that $G$ acts on~$X$ with an open orbit ${{\mathscr{O}}_G}$, see ~\cite[Corollary~1.11a]{AFKKZ}. 
\end{sit}

The following theorem is the main result of this section.

\begin{thm}\label{th-inf-tr} If one of $(\alpha)$-$(\gamma)$ holds then
the action of~$G$ on ${{\mathscr{O}}_G}$ is infinitely transitive.
\end{thm}

This theorem is actually a refined version of Theorem~2.2 in~\cite{AFKKZ}. The proof follows, with some modifications, the lines of the proof of Theorem~2.2 in~\cite{AFKKZ}. 

\begin{lem}\label{lemma1-31mars}
Let $\Omega\subset {{\mathscr{O}}_G}$ be a dense open subset. 
Then for any finite collection of distinct points $Q_1,Q_2,\ldots,Q_m\in {{\mathscr{O}}_G}$ there 
exists $g\in G$ such that  
$g(Q_i)\in \Omega$ for every $i=1,2,\ldots,m$.
\end{lem}
\begin{proof}
By~\cite[Prop.\ 1.5]{AFKKZ} there is a finite ordered collection of (not necessarily distinct)
$\GG_a$-subgroups $U_1,U_2,\ldots,U_N$ in~$G$ such that for any~$x\in {{\mathscr{O}}_G}$ 
we have 
${{\mathscr{O}}_G}=(U_1\cdot\ldots\cdot  U_N).x$. This gives a surjective morphism $$\varphi_x\colon \AA^N\rightarrow {{\mathscr{O}}_G},\quad (t_1,\ldots,t_N)\mapsto (U_1(t_1)\cdot\ldots\cdot  U_N(t_N)).x\,.$$
Letting $\varphi_i=\varphi_{Q_i}$, $i=1,\ldots,m$ consider the dense open subset 
$$
\omega=\bigcap_{i=1}^m\varphi_i^{-1}(\Omega)
\subset\AA^N\,.$$ Pick up a point $(t_1,t_2,\ldots,t_N)\in \omega$, and let $g=U_1(t_1)\cdot\ldots\cdot U_N(t_N)\in G$. Then for any $i=1,\ldots,m$
one has $g(Q_i)\in\Omega$.
\end{proof}

In the sequel we use the following notation.

\begin{nota}\label{not:omega} Let $H_i=H_i(1)$, $i=1,\ldots,k$.
Letting $W_i=\Spec A_i$
consider the morphism $\pi_i\colon X\to W_i$ induced by the inclusion 
$A_i\hookrightarrow \cO_X(X)$. There is a Zariski open, dense subset $\omega_i\subset W_i$ such that on $U_i:=\pi_i^{-1}(\omega_i)$ 
there exists the geometric quotient $U_i/H_i$. The inclusion $A_i\subset\cO_{U_i}(U_i)$ induces a generically finite morphism $q_i\colon U_i/H_i\to W_i$. 
Due to our assumption, $\omega_i$ can be chosen so that $q_i$ is a finite morphism  onto its image, of degree $d_i:=[{\rm Frac}\,(\ker\partial_i):{\rm Frac}\,(A_i)]$.
One can find
a dense open subset $\Omega\subset {{\mathscr{O}}_G}$ such that 
\begin{itemize}
\item[(i)]   in each point $x\in\Omega$ the vectors  $\partial_1(x),\ldots,\partial_k(x)$
are pairwise non-collinear and generate  the tangent space $T_xX$;
\item[(ii)]  for $i=1,\ldots,k$ one has $\pi_i(\Omega)\subset {\reg} (W_i)$ and the restriction  $\pi_i\vert_\Omega\colon \Omega\to W_i$ is a smooth morphism;
\item[(iii)]  for $i=1,\ldots,k$ the fiber of $\pi_i\vert_\Omega$ over any point $w\in\pi_i(\Omega)$
is a dense open subset of the union of $d_i$ orbits of $H_i$;
\item[(iv)] in case ($\gamma$), for $i=1$ these orbits are separated by the function $b_1\in\ker\partial_1$ as in~\ref{sit-2.1}.
\end{itemize}
For each $i=1,\ldots,k$ there is a factorization
$$\pi_i\colon \Omega\stackrel{p_i}{\longrightarrow}\Omega/H_i\stackrel{q_i}{\longrightarrow} W_i\,.$$
\end{nota}

We have the following analogue of Lemma 2.10 in~\cite{AFKKZ}.

\begin{lem}\label{two points} 
For any pair of distinct points $Q_1,Q_2\in {{\mathscr{O}}_G}$ there exists 
$g\in G$ such that for every~$i=1,\ldots,k$ the points 
$g.Q_1$ and $g.Q_2$ are separated by 
$A_i$ for $i=1,\ldots,k$. 
\end{lem}

\begin{proof} By Lemma~\ref{lemma1-31mars} one may suppose that $Q_1,Q_2\in \Omega$ where $\Omega\subset X$ is as in~\ref{not:omega}.
Assume first that for some $i\in \{1,\ldots,k\}$ the algebra $A_i$ separates $Q_1$ and $Q_2$. This is so in cases ($\alpha$) and ($\beta$); 
anyway, one may consider that $i=1$. Let $a_1\in A_1$ be such that $a_1(Q_1)=0$ and $a_1(Q_2)=1$. Then $H_1(a_1)$ fixes $Q_1$ 
and moves $Q_2$ along its $H_1$-orbit. Let 
$$Q_{2}(t)=\exp(t \partial_1)(Q_2),\quad t\in\kk\,.$$ Given $i\ge 2$ the condition 
\begin{equation}
\label{eq:cond} Q_{2}(t)\in\Omega\quad\mbox{and}\quad\pi_i(Q_{2}(t))\neq\pi_i(Q_{1})
\end{equation}
is an open condition on~$t\in\kk$. Since $Q_2\in\Omega$, by (i) and (ii) 
the image $\pi_i(H_1(Q_2))$ in $W_i$ is one-dimensional. It follows that \eqref{eq:cond} holds on a dense open subset in $\kk$. 
Moreover, the latter is true simultaneously for all $i=2,\ldots,k$, as required.

Now one may restrict to case ($\gamma$). Suppose that $A_1$ does not separate $Q_1$ and $Q_2$.  
Assume further that $H_1(Q_1)\neq H_1(Q_2)$. Then $b_1$ separates $Q_1$ and $Q_2$ due to (iv). 

Consider the flow
$$\phi_t=\exp(t(b_1-b_1(Q_1))\partial_1)\subset H_0\cdot H_1\subset G\,,$$  
and let $Q_2(t)=\phi_t(Q_2)$. Then $\phi_t$ fixes $Q_1$ and moves $Q_2$ along its $H_1$-orbit. 
Applying the same argument as before one can see that that for every 
$i=2,\ldots, k$ the algebra $A_i$, $i=2,\ldots,k$ 
separates $Q_1$ and $Q_2(t)$ for a general $t\in\kk$. Since by our assumptions $k\ge n\ge 2$, 
one may interchange now the role of $A_1$ and $A_k$ and achieve as before that $A_1$ 
separates the images of $Q_1$ and $Q_2$ under the action of $\exp(ta_k\partial_k)(Q_1)$ 
for a suitable $a_k\in A_k$ and a general $t\in\kk$. This gives the result.

Suppose further that  $H_1(Q_1)=H_1(Q_2)$. We claim that for every~$i=2,\ldots,k$ 
and for a general $t\in\kk$ the points $Q_{1}(t)$ and $Q_{2}(t)$ are separated by $A_i$. 
 Indeed, assume to the contrary that $\pi_i(Q_{1}(t))=\pi_i(Q_{2}(t))$ for some $i\ge 2$ and for all $t\in\kk$. Since
the image $\pi_i(H_1(Q_1))$ in $W_i$ is one-dimensional there exists $a_i\in A_i$ 
such that the restriction $a_i|_{H_1(Q_1)}$ defines a non-constant polynomial  $p_i\in\kk[t]$. 
Since $Q_2=Q_{1}(\tau)$ for some nonzero $\tau\in\kk$ one has $Q_{2}(t)=Q_{1}(t+\tau)$. 
It follows that $p_i(t)=p_i(t+\tau)$ for any $t\in\kk$, a contradiction. The proof ends by the argument  used in the previous case for $i=1$.
\end{proof}

\begin{lem}\label{different-levels}
For any  finite collection of distinct points  $Q_1,\ldots,Q_m\in {{\mathscr{O}}_G}$ 
there exists an element $g\in G$ such that  
the points $g(Q_1),\ldots,g(Q_m)$ are separated by $A_i$ for $i=1,\ldots,k$. 
\end{lem}
\begin{proof}
By Lemma~\ref{lemma1-31mars} we may assume $Q_j\in \Omega$ $\forall i=1,\ldots,m$.  
 We proceed by induction on $m$. For $m=1$ the assertion is evidently true. Assume that 
 the points $Q_1,\ldots,Q_{m-1}$ are already separated by $A_i$ for $i=1,\ldots,k$. 
 Applying   Lemma~\ref{two points} and its proof to $Q_m$ and $Q_1$ one may replace the cortege  
 $(Q_1,\ldots,Q_m)$ by a new one  $(Q^{(1)}_1(t),\ldots, Q^{(1)}_m(t))$ so that the separation property
holds for $Q^{(1)}_1(t),\ldots, Q^{(1)}_{m-1}(t)$  with a generic $t\in\kk$,  
 and,  in addition, for $Q^{(1)}_1(t)$ and $Q^{(1)}_m(t)$. Fixing such a value $t_1\in\kk$ 
 one may apply the same procedure to obtain a new cortege $(Q^{(2)}_j(t_2))_{j=1,\ldots,m}$ 
 preserving the former property and adding the separation of $Q^{(2)}_2(t_2)$ and $Q^{(2)}_m(t_2)$, and so for. 
Finally one arrives at a cortege with the desired separation property.
\end{proof}

\begin{lem}\label{prop1}
For any  finite collection of distinct points  $Q_1,\ldots,Q_m\in {{\mathscr{O}}_G}$ the stabilizer 
$\Stab_{Q_1,\ldots,Q_m}(G)$
acts transitively on ${{\mathscr{O}}_G}\setminus \{Q_1,\ldots,Q_m\}$.
\end{lem}

\begin{proof} We proceed by induction on $m$. The assertion is evidently true for $m=0$.
Assuming it holds for a given $m\ge 0$ consider a collection of $m+1$ distinct points 
$Q_1,\ldots,Q_m, Q_{m+1}\in {{\mathscr{O}}_G}$. By Lemma~\ref{lemma1-31mars} 
one may assume these points lie in $\Omega$. Applying Lemma~\ref{different-levels} one may suppose that  for $i=1,\ldots,k$
the images $\pi_i(Q_j)\in W_i$, $j=1,\ldots,m+1$ are all distinct. 
Then for every $i=1,\ldots,k$ there exists $a_i\in A_i$ which vanishes at  $Q_1,\ldots,Q_m$ and does not vanish at $Q_{m+1}$. 

Consider the $\GG_a$-subgroups $$H_i(a_i)\subset {\Stab}_{G}(Q_1,\ldots,Q_m),\quad i=1,\ldots,k\,.$$ 
The orbit of $Q_{m+1}$ under the action of the stabilizer ${\Stab}_{G}(Q_1,\ldots,Q_m)$  is locally closed 
(see, e.g., Proposition 1.3 in \cite{AFKKZ}) and contains the one-dimensional $H_i$-orbits of $Q_{m+1}$, $i=1,\ldots,k$. 
By (i) the 
tangent vectors to these orbits at $Q_{m+1}$ span the tangent space $T_{Q_{m+1}}X$. It follows that the orbit 
${\Stab}_{G}(Q_1,\ldots,Q_m)(Q_{m+1})$ is open in~$X$ whatever is the point $Q_{m+1}\in {{\mathscr{O}}_G}\setminus \{Q_1,\ldots,Q_m\}$. 
Since an open dense orbit  is unique one has
$${\Stab}_{G}(Q_1,\ldots,Q_m)(Q_{m+1})={{\mathscr{O}}_G}\setminus \{Q_1,\ldots,Q_m\}\,.$$
\end{proof}

\begin{proof}[Proof of Theorem~{\rm~\ref{th-inf-tr}}] 
We have to show that for any two ordered corteges $(Q_1,\ldots,Q_m)$ and $(Q'_1,\ldots,Q'_m)$ in ${{\mathscr{O}}_G}$ 
there is $g\in G$ such that $g.Q_j=Q_j'$, $j=1,\ldots,m$. Assuming by induction $Q_i=Q'_i$, $i=1,\ldots,m-1$, 
by Lemma~\ref{prop1} one can find~$g\in{\Stab}_{G}(Q_1,\ldots,Q_{m-1})$ 
such that $g.Q_m=Q_m'$, as required.  
\end{proof}

\begin{cor}\label{cor:countable}
Let $X$ be a generically flexible affine variety \footnote{That is, the group $\SAut(X)$ acts on $X$ with an open orbit.} 
of dimension $n\ge 2$. Then there exists a countable collection of $\GG_a$-subgroups $\{H_1,\ldots,H_n,\ldots\}$ 
such that the subgroup $$G=\langle H_i\,|\,i\in\NN\rangle\subset\SAut(X)$$ acts on $X$ with an open orbit ${\mathscr{O}}_G$ 
and is infinitely transitive on ${\mathscr{O}}_G$.
\end{cor}

\begin{proof} 
The  generic flexibility of $X$ implies that there is a collection of $n$ linearly independent LNDs 
$\p_1,\ldots,\p_n\in {\LND}\,(X)$. Letting $H_i=\exp(\kk\p_i)$ choose for any $i=1,\ldots,n$ 
a finitely generated subalgebra $A_i\subset\ker\p_i$  separating the general $H_i$-orbits, 
and let  $\{a_{i,j}\}_{j\in\NN}$ be a countable Hamel basis of $A_i$ viewed as a vector space over $\kk$. 
Letting $H_{i,j}=\exp(\kk a_{i,j}\p_i)$ consider the group 
$$G=\langle H_{i,j}\,|\,i=1,\ldots,n,\,j\in\NN\rangle\subset\SAut(X)\,.$$ Clearly, $G$ acts on $X$ 
with an open orbit ${\mathscr{O}}_G$ and satisfies condition ($\beta$) of~\ref{sit-2.1}. 
Therefore, applying Theorem~\ref{th-inf-tr} one arrives at the desired conclusion. 
\end{proof} 

\begin{exa}\label{ex:toric}
Let $X$ be a toric affine variety of dimension $\ge 2$  with no torus factor.
Then the subgroup $G\subset\SAut(X)$ generated by all the root subgroups $H_e$, 
where $e$ runs over the (countable) set of all the Demazure roots of $X$ 
(see Section~\ref{sec:toric}) acts infinitely transitively on $\reg(X)$. 
This follows from Corollary~\ref{cor:countable} or, alternatively, from the proof of Theorem 2.1 in \cite{AKZ}.
\end{exa}

\section{Group closures and orbits}\label{sec:orbits}
We gather some facts that will be used in the next section.

\begin{sit}\label{sit:3.1} Recall that  $\Aut(X)$ has a structure of an affine ind-group; see, e.g., \cite{FK, Kum} for generalities. 
In more detail, following \cite[Prop.\ 2.1]{KPZ} we fix an embedding $X\hookrightarrow\A^n$ 
and introduce in $\mathcal{O}_X(X)$ a (positive) degree function. For $\alpha\in\Aut(X)$ one defines $\deg(\alpha)$ 
to be the maximum of the degrees of components of $\alpha$. One can write $\Aut(X)=\indlim \Sigma_s$ where 
\begin{itemize} 
\item
for  $s\ge 1$, $\Sigma_s:=\{\alpha\in\Aut(X)\,|\,\deg (\alpha), \deg(\alpha^{-1})\le s\}$ is a closed subvariety of the affine variety $\Sigma_{s+1}$;
\item  for any $r,s\ge 1$ the composition yields a morphism $\Sigma_r\times\Sigma_s\to\Sigma_{rs}$;
\item the inversion yields an automorphism of $\Sigma_s$. 
\end{itemize}
The Zariski closure of a subset $F\subset\Aut(X)$ can be defined as $$\overline{F}=\indlim \overline{(F\cap\Sigma_s)}\,$$ 
where the overline stands for the Zariski closure in $\Sigma_s$. 
The Zariski closure of $F$ is a closed ind-subvariety of the ind-variety $\Aut(X)$.  
An \emph{algebraic subgroup} of $\Aut(X)$ is a subgroup which is a closed subvariety of some $\Sigma_s$.
\end{sit}

\begin{lem}\label{lem:orbit-closures} \begin{itemize}
\item[{\rm (a)}] The closure $\overline{G}$ of a subgroup $G\subset\Aut(X)$ is a closed ind-subgroup of $\Aut(X)$.
\item[{\rm (b)}] 
If $\rho\colon\mathbb{A}^1\to\Aut(X)$ is a morphism such that $\rho(t)\in G$ for $t\neq 0$ then $\rho(0)\in\overline{G}$.
\item[{\rm (c)}] Any $G$-invariant closed subset $Y\subset X$ is $\overline{G}$-invariant. 
\item[{\rm (d)}] If $G$ acts on $X$ with an open orbit ${{\mathscr{O}}_G}$ then  ${{\mathscr{O}}_G}$ 
coincides with the open orbit ${{\mathscr{O}}_{\overline{G}}}$ of $\overline{G}$.
\item[{\rm (e)}] If a normal subgroup $G\subset\Aut(X)$ acts on $X$ with an open orbit ${{\mathscr{O}}_G}$ 
then ${{\mathscr{O}}_G}={{\mathscr{O}}_{\Aut(X)}}$.
\end{itemize}
\end{lem}

\begin{proof}
(a) Let $G_s=G\cap\Sigma_s$. By definition, $\overline{G}_s=\overline{G}\cap \Sigma_s$.
Since $G_r\cdot G_s\subset G_{rs}$ and $\Sigma_r\times \Sigma_s\to \Sigma_{rs}$ 
is a morphism then $\overline{G}_r\times \overline{G}_s\to \overline{G}_{rs}$ is a morphism. 
Since $\Sigma_s^{-1}=\Sigma_s$ and the inversion is an automorphism of $\Sigma_s$ one has 
$G_s^{-1}=G_s$ and the inversion $G_s\to\Sigma_s$ extends to a morphism $\overline{G}_s\to\Sigma_s$ 
which is still the inversion with values in $\overline{G}_s$. 
Now (a) follows. 

(b) One has $\rho(\mathbb{A}^1)\subset\Sigma_s$ for some $s\ge 1$. 
Hence $\rho(\mathbb{A}^1\setminus\{0\})\subset G_s$, and so, $\rho(0)\in\overline{G}_s$.

(c) The statement follows immediately from the fact that the action map $\Sigma_s\times X\to X$ is a morphism for any $s\ge 1$. 

(d) 
Suppose to the contrary that 
${{\mathscr{O}}_G}\subsetneq {{\mathscr{O}}_{\overline{G}}} $. Then
$Y={{\mathscr{O}}_{\overline{G}}}\setminus {{\mathscr{O}}_G}$ is a nonempty proper 
$G$-invariant closed subset of ${{\mathscr{O}}_{\overline{G}}}$. By (c), $Y$ is $\overline{G}$-invariant, a contradiction.

The statement of  (e) is a simple exercise. 
\end{proof}

\begin{sit}
A subgroup $G\subset\Aut(X)$ generated by a family of connected algebraic subgroups of $\Aut(X)$ 
is called \emph{algebraically generated} (\cite{AFKKZ}). The orbits of $G$ are locally closed subsets of $X$ 
in the Zariski topology; see \cite[Prop.\ 1.3]{AFKKZ}.  
\end{sit}

\begin{prop}\label{lem:many-points}
Let $G\subset \Aut(X)$ be an algebraically generated subgroup. Then the following hold.
\begin{itemize}
\item[{\rm (a)}] The orbits of $G$ and of $\overline{G}$ in $X$ are the same.
In particular, if  $\overline{G}$ acts on $X$ with an open orbit ${{\mathscr{O}}_{\overline{G}}}$ then $G$ 
does and ${{\mathscr{O}}_{G}}={{\mathscr{O}}_{\overline{G}}}$.
\item[{\rm (b)}]  If  $\overline{G}$ acts $m$-transitively on ${{\mathscr{O}}_{\overline{G}}}$ then also $G$ does.
\item[{\rm (c)}]  If  $\overline{G}$ acts  infinitely transitively on ${{\mathscr{O}}_{\overline{G}}}$ then also $G$ does.
\end{itemize}
\end{prop}

\begin{proof}
(a) Let $x\in X$, and let $Y=\overline{G.x}$. By Lemma~\ref{lem:orbit-closures}(c), $Y$ is $\overline{G}$-invariant. 
The orbits $G.x$ and ${\overline{G}}.x\supset G.x$ are both open and dense subsets of $Y$. Suppose to the contrary 
that ${\overline{G}}.x\neq G.x$, and let $Z=Y\setminus G.x$. Then $Z$ is a nonempty $G$-invariant
closed subset of $X$. It is $\overline{G}$-invariant  by Lemma~\ref{lem:orbit-closures}(c), and, by our assumption, 
meets the orbit ${\overline{G}}.x$ which is open  in $Y$. 
Hence $Z\supset{\overline{G}}.x\supset G.x$. This is a contradiction.

(b) Choose a cortege $\mathcal{Q}$ of distinct points $Q_1,\ldots,Q_m\in {{\mathscr{O}}_{\overline{G}}}$. 
Consider  the diagonal action of $\Aut(X)$ on $X^m$, and let $D\subset X^m$ be  the union of all big diagonals. 
Assume that $\overline{G}$ acts $m$-transitively on ${{\mathscr{O}}_{\overline{G}}}$. 
Then the $\overline{G}$-orbit of $\mathcal{Q}$ in $X^m\setminus D$ coincides with $({{\mathscr{O}}_{\overline{G}}})^m\setminus D$. 
In particular, it is open. The image of  $\overline{G}$  in $\Aut(X^m)$ is contained in the closure of the image of $G$ in $\Aut(X^m)$. 
Applying (a) one concludes that $G.\mathcal{Q}=\overline{G}.\mathcal{Q}=({{\mathscr{O}}_{\overline{G}}})^m\setminus D$, 
that is, $G$  acts $m$-transitively on ${{\mathscr{O}}_{\overline{G}}}$.

Finally, (c) is immediate from (b).
\end{proof}

\section{Toric varieties, Cox rings, and derivations}\label{sec:toric}
\subsection{Toric affine varieties and Demazure roots}
\begin{sit}\label{sit:Dem-roots}
Recall the combinatorial description of a toric affine variety
(see, e.g.,~\cite[Ch.\ 1]{CLS}, \cite[Sec.\ 1.3]{Ful}). Let $M$ be a lattice of rank~$n$, let $M_\QQ=M\otimes\QQ$ 
be the associated vector space over $\QQ$,
and let $$\sigma^\vee\subset M_{\QQ}$$ be a rational convex cone with a nonempty interior (the \emph{weight cone}). 
Fix a basis of $M$. For  $m=(m_1,\ldots,m_n)\in M$ by $\chi^m$ one means a Laurent monomial $x_1^{m_1}\ldots x_n^{m_n}$. 
Consider the graded affine algebra $$A=\bigoplus _{m\in M\cap \sigma^\vee}\kk\chi^m\,.$$ Then $X=\Spec A$ 
is a toric affine  variety of dimension $n$ equipped with the $n$-torus action defined by the grading. 
In fact, any toric affine variety arises in this way. The acting algebraic torus is the torus of characters $\T={\Hom}(M,\GG_m)$.  
By duality, $M$ is the character lattice of~$\T$. 

Consider the dual lattice $N={\Hom}(M, \ZZ)$  and the dual cone 
$$
 \sigma \subset N_{\QQ}, \quad \sigma=\{x\in N_{\QQ} \,|\, \langle x,y\rangle\ge  0\,\,\, \forall y\in \sigma^\vee\}\,.
 $$
A {\it ray generator} of~$\sigma$ is a primitive lattice vector on an extremal ray of~$\sigma$. 
Let ~$\Xi$ be the set of  ray generators of~$\sigma$.  Assume $X$ has no torus factor, that is, 
$X$ cannot be decomposed into a product ${\mathbb G}_m \times Y$ where $Y$ is another toric variety. 
The latter is equivalent to the fact that the cone $\sigma^\vee$ is pointed, that is, contains no line, and also 
to the fact that $\sigma$ is of full dimension, that is, $\Xi$ contains a basis of $N_{\QQ}$.
To any vector $\rho\in N$ there corresponds a $\GG_m$-subgroup $R_\rho\subset \T$ acting via 
$$
t.\chi^m=t^{\langle \rho,m\rangle}\chi^m,\quad t\in {\kk^*}, \,\,m\in  \sigma^\vee\cap M\,.
$$ 
\end{sit}

\begin{defi}[\emph{Demazure roots and Demazure facets}] \label{def:Demazure-root} 
Let~$X=\Spec A$ be a toric affine variety with no torus factor
associated to a lattice cone $\sigma^\vee\subset M_\QQ$, and let $\Xi=\{\rho_1,\ldots,\rho_k\}$ 
be the set of primitive ray generators of $\sigma\subset N_\QQ$. A \emph{Demazure root}  
which belongs to a primitive ray generator $ \rho_i\in\Xi$ is a vector $e\in M$ such that 
\begin{itemize}\item[{\rm (i)}]
$\langle \rho_i,e \rangle=-1$;
\item[{\rm (ii)}]
$\langle \rho_j,e \rangle\geqslant 0$ $\forall j\neq i$,
\end{itemize}
see \cite[\S 3.1]{Dem}, \cite{Li, Lie}.  The rational convex polyhedron  $\mathcal{S}_{i}$ 
defined in the affine hyperplane $\mathcal{H}_{i}=\{\langle \rho_i,e \rangle=-1\}$ by (ii) will be called 
a \emph{Demazure facet} of $\sigma^\vee$.  The Demazure roots belonging to the ray generator $ \rho_i\in\Xi$  are the points in 
$\mathcal{S}_{i}\cap M$.
\end{defi}

\begin{defi}[\emph{Homogeneous derivations}] A derivation $\partial\in\Der(A)$ is called \emph{homogeneous} 
if $\partial$ respects the grading, that is, sends any graded piece to another one. \end{defi}

The following description of homogeneous derivations on toric affine varieties completes the one in \cite[Sect.\ 2]{Lie}. 

\begin{prop}\label{prop:homog-deriv} \begin{itemize}
\item[{\rm (a)}]
Any homogeneous derivation $\partial\in\Der(A)$ has the form $\partial=\lambda \partial_{\rho,e}$ 
for some $\lambda\in\kk$, $\rho\in N$, and $e\in M$ where
\begin{equation}\label{eq:homog-deriv} \partial_{\rho,e}(\chi^m)=\langle \rho,m \rangle\chi^{m+e} \quad\forall m\in\sigma^\vee\cap M\,.
\end{equation} {\rm (The lattice vector $e$ is called the \emph{degree} of $\partial$.)}
\item[{\rm (b)}] Let $$\Sigma^\vee=\sigma^\vee\cup\bigcup_{i=1}^k\mathcal{S}_i\,.$$
Then
$\partial_{\rho,e}(A)\subset A$ if and only if $e\in\Sigma^\vee\cap M$ and either $e\in\sigma^\vee$, or  
$e\in\mathcal{S}_i\cap M$ and $\rho=\rho_i$.
\item[{\rm (c)}] {\rm (\cite[Lem.\ 2.6 and Thm.\ 2.7]{Lie})}
A homogeneous  derivation $\partial\in {\Der}(A)$ is locally nilpotent if and only if $\partial=\lambda\partial_{\rho_i,e}$
 for  a Demazure root $e\in \mathcal{S}_{i}$ and  for some $\lambda\in\kk$. 
\end{itemize}
\end{prop}

The proof is straightforward.

\begin{rem} \label{rem:4.5}
The kernel of $\partial_{\rho,e}$ is spanned by the characters~$\chi^m$ where~$m\in M$ 
belongs to the hyperplane section $\tau_\rho$ of $\sigma^\vee$ defined by $\langle \rho, m \rangle =0$. 
If $e\in\mathcal{S}_i\cap M$ is a Demazure root  then $\tau_i:=\tau_{\rho_i}=\rho_i^\vee$
is a  facet of $\sigma^\vee$. The affine hyperplane $\mathcal{H}_i$ spanned by the Demazure facet $\mathcal{S}_i$ is parallel to $\tau_i$.
\end{rem}

\begin{defi}[\emph{Root subgroups}]\label{def:root-sbgrp} Given a Demazure root $e\in\mathcal{S}_i$ the associated 
one-parameter unipotent subgroup $H_e=\exp(\kk\partial_{\rho_i,e})\subset\SAut(X)$ is called a \emph{root subgroup}. 
\end{defi}

The next lemma is mainly borrowed in \cite[Lem.\ 1.10]{Lie}.

\begin{lem} \label{lem:non-homog-deriv} \begin{itemize}
\item[\rm (a)] Any derivation $\p\in\Der(A)$ admits a decomposition \begin{equation}\label{eq:decomp}
\p= \sum_{e\in \Sigma^\vee\cap M} \p_e\end{equation} 
where $\p_e$ is a homogeneous derivation of degree $e$. 
\item[\rm (b)] The set $\{e \in \Sigma^\vee\cap M | \p_e\neq 0\}$ is finite. Its convex hull $N(\p)$ 
is a polytope called the {\rm Newton polytope of} $\p$. 
\item[\rm (c)] If $\p\in {\LND}\,(A)$  then for any face $\tau$ of $N(\p)$ one has
$$\p_\tau:= \sum_{e\in \tau\cap M} \p_e\in {\LND}\,(A)\,.$$ In particular, for any vertex $e$ of $N(\p)$ one has $\p_e\in {\LND}\,(A)$. 
\end{itemize}
\end{lem}

\begin{proof}
To show (c) it suffices to notice that $$(\p_\tau)^l(\chi^m)=(\p^l)_{l\tau}(\chi^m)=
0\quad\forall m\in\sigma^\vee\cap M\quad\mbox{and}\quad \forall l=l(m)\gg 1\,.$$ For the rest see the proof of \cite[Lem.\ 1.10]{Lie}.
\end{proof}

The following lemma is immediate.

\begin{lem}\label{lem:LND-comm} The semigroup $\mathcal{S}_i\cap M$ is a finitely generated $(\tau_i\cap M)$-module. 
For any $f\in \tau_i\cap M$ and any $e\in\mathcal{S}_i\cap M$ one has $$\chi^f\p_{\rho_i,e}=\p_{\rho_i,e+f}\in {\LND}\,(A)\,.$$ 
\end{lem}

\subsubsection{Iterated commutators on toric varieties}
For homogeneous derivations one has the following lemma (cf.\ \cite[Prop.\ 1 and Lem.\ 2]{Lena}).

\begin{lem}\label{commutator} \begin{itemize}\item[\rm (a)]
For two nonzero homogeneous derivations $\partial=\partial_{\rho,e},\,\partial'=\partial_{\rho',e'}\in\Der(A)$ one has
\begin{equation}\label{eq:commutator}
[\partial, \partial']=\partial_{\hat\rho,\hat e}\quad\mbox{where}\quad
\hat\rho=\langle \rho,e' \rangle\rho'-\langle \rho',e \rangle\rho\in N\quad\mbox{and}\quad\hat e=e+e'\,.
\end{equation}
\item[\rm (b)]
If $\hat\rho\neq 0$ then $[\partial, \partial']\in\Der(A)$ is a homogeneous derivation of degree $e+e'$ with the linear form~$\hat\rho$.  
In particular, $e+e'\in \Sigma^\vee\cap M$. 
\item[\rm (c)] $\hat\rho= 0$ {\rm (}that is, $\partial$ and $\partial'$ commute{\rm )} if and only if
one of the following holds:
\begin{itemize}
\item $\rho$ and $\rho'$ are collinear and 
$\langle \rho, e\rangle = \langle \rho, e'\rangle$ {\rm (}this holds, in particular, if $e,e'\in\mathcal{S}_i$ for some $i\in\{1,\ldots,k\}${\rm )};
\item $\rho$ and $\rho'$ are non-collinear and $\langle \rho', e\rangle = \langle \rho, e'\rangle=0$.
\end{itemize}
\end{itemize}
\end{lem}

\begin{sit}\label{notation} Given two derivations (or vector fields on $X$) $U=\partial_1$ and $V=\partial_2$ in $\Der(A)$ we let 
\begin{equation}\label{eq:ad-m} {\ad}_U^m(V)=[U,[U,\ldots [U,V]\ldots ]]=\sum_{i=0}^{m}{{m}\choose{i}}\p_1^{m-i}\p_2(-\p_1)^i\end{equation}
where $U$ is repeated $m$ times, see \cite{Ma}. Thus, ${\ad}_U^m\in {\rm End}\,(\Der(A))$.
\end{sit}

\begin{lem}[cf.\ \cite{Ma}]\label{lem:ad-LND}
If $U\in\LND(A)$ then ${\ad}_U\in {\rm End}\,(\Der(A))$ is locally nilpotent, that is,
 for any $V\in\Der(A)$ there exists $c(V)>0$ such that ${\ad}_U^m(V)=0$ $\forall m>c(V)$.
\end{lem}

\begin{proof}  Let $A=\kk[a_1,\ldots,a_k]$. There exists $N>0$ such that $\p_1^N(a_j)=0$ $\forall j=1,\ldots,k$. 
In the last sum in \eqref{eq:ad-m} applied  to $a_j$ certain members of the sum vanish so that it remains just $N$
 first members of the sum.  For $m\gg 1$ the first $N$ terms vanish as well, hence 
${\ad}_U^m(V)(a_j)=0$ for $j=1,\ldots,k$, and so, ${\ad}_U^m(V)=0$.
\end{proof}

\begin{sit}\label{sit:BCH} We keep $U=\p_1\in\Der(A)$ being locally nilpotent. In the sequel 
we use the following version of the Baker-Campbell-Hausdorff formula, see, e.g., \cite[Thm.\ 4.14]{BF}:
\begin{equation}\label{eq:BCH} 
{\Ad}_{\exp (U)}(V) 
=\exp({\ad}_U)(V)=V+\sum_{m=1}^\infty \frac{1}{m!}{\ad}_U^m(V)\,.\end{equation}
Due to Lemma~\ref{lem:ad-LND}, for a given $V=\p_2$ the last sum is finite, hence the formula has sense. 
The second equality in \eqref{eq:BCH} is just the usual exponential formula for the endomorphism ${\ad}_U$ 
of the vector space $\Der(A)$. Due to the first equality in \eqref{eq:BCH}, if $V$ is locally nilpotent then $\exp({\ad}_U)(V)$ is as well. 
\end{sit}

To apply this formula in our setting we need the following simple lemma. 

\begin{lem}\label{lem:iterated} Consider two nonzero homogeneous LNDs, $U=\partial_{\rho_1,e_1}$ 
and $V=\partial_{\rho_2,e_2}$ in $\Der(A)$ where $e_i\in \mathcal{S}_{i}\cap M$, $i=1,2$.  
Letting $$c_2=\langle \rho_2, e_1\rangle, \quad d_1=\langle \rho_1,e_2\rangle,\quad\mbox{and}\quad \delta=d_1+1$$ 
and assuming  $c_2\ge 1$ one has
$${\ad}_U^m(V)=\p_{r_m,f_m}\quad \forall m=0,\ldots,d_1$$ where
\begin{equation}\label{eq:iterate} r_m=\frac{d_1!}{(d_1-m)!} \rho_2-mc_2\frac{d_1!}{(d_1-m+1)!}\rho_1
\quad\mbox{and}\quad
f_m=e_2+m e_1\,.\end{equation}
Furthermore, if $d_1\ge 0$ then one has 
\begin{equation}\label{eq:vanishing}
{\ad}_U^m(V)=0\quad\forall m\ge \delta+1\end{equation} and 
\begin{equation}\label{eq:delta-iterate}
{\ad}_U^\delta(V)=-c_2\delta!\p_{\rho_1,f_\delta}\in\LND(A)\quad\mbox{where}\quad 
f_\delta=e_2+\delta e_1\in\mathcal{S}_{1}\cap M\,.\end{equation}
\end{lem}

\begin{proof}
 The  assertions follow immediately from \eqref{eq:commutator} by recursion on $m$. 
\end{proof}

\begin{cor}\label{cor:NP-Ad} Under the assumptions of Lemma {\rm\ref{lem:iterated}} let $\p={\Ad}_{\exp(U)}(V)$. Then 
the Newton  polytope $N(\p)$ is the segment
$[e_2,e_2+\delta e_1]$. 
\end{cor}

\begin{proof}
Applying formula \eqref{eq:BCH}, by Lemma~\ref{lem:iterated}  one obtains:
$$\p=V+\sum_{m=1}^\delta \frac{1}{m!} \p_m$$
where $\p_m=\p_{r_m,f_m}\in\Der(A)$ is a homogeneous derivation of degree
$f_m=e_2+me_1$. Now the result follows.
\end{proof}

\subsubsection{$\Der(A)$ as a graded Lie algebra} 
Given a lattice vector $e\in\Sigma^\vee$ consider the linear subspace $${\Der}_e(A)={\span}\{\partial_{\rho,e}\}_{\rho\in N}\subset\Der(A)$$ 
generated by the homogeneous derivations of $A$ of degree $e$. By Lemma~\ref{commutator} one has 
\[[{\Der}_e(A),{\Der}_{e'}(A)]\subset{\Der}_{e+e'}(A)\,. \]
Hence, the Lie algebra $\Der(A)$ is $M$-graded:
$$ \Der(A)=\bigoplus_{e\in M} {\Der}_e(A)\quad\mbox{where}\quad {\Der}_e(A)\neq\{0\}\Leftrightarrow e\in\Sigma^\vee\cap M\,,$$ 
see Proposition~\ref{prop:homog-deriv}(b) and Lemma~\ref{lem:non-homog-deriv}(a). 

Any subgroup $T$ of the torus $\T$ acts by conjugation on $\Der(A)$, and so, induces a grading of $\Der(A)$ 
compatible with the $M$-grading, see, e.g., \cite{Kot}.

\begin{exa}
Consider the standard action of the 2-torus $\T$ on $\A^2=\Spec\kk[x,y]$. This action induces an 
$\N^2$-grading on $\Der(\kk[x,y])$ with 
graded pieces ${\Der}_e(\kk[x,y])$ where $e=(i,j)\in\ZZ^2$ runs over the lattice vectors with $i,j\ge -1$ and $(i,j)\neq (-1,-1)$. 
The piece ${\Der}_e(\kk[x,y])$ consists of all the homogeneous derivations of $\kk[x,y]$ of degree $e$. For $i,j\ge 0$ one has
$${\Der}_e(\kk[x,y])=\left\{x^iy^j\left(ax\frac{\p}{\p x}+by\frac{\p}{\p y}\right)\,|\,a,b\in\kk\right\}\,.$$ 
The pieces which correspond to the Demazure facets \begin{equation}\label{eq:facets}
\mathcal{S}_1\cap M=\{(-1,j)\,|\,j\in\ZZ_{\ge 0}\}\quad\mbox{and}\quad\mathcal{S}_2\cap M=\{(i,-1)\,|\,i\in\ZZ_{\ge 0}\}\end{equation} 
 are
$${\Der}_{(-1,j)}(\kk[x,y])=\left\{ay^j\frac{\p}{\p x}\,|\,a\in\kk\right\}\quad\mbox{resp.,}
\quad{\Der}_{(i,-1)}(\kk[x,y])=\left\{ax^i\frac{\p}{\p y}\,|\,a\in\kk\right\}\,.$$
Any one-parameter subgroup $T\subset\TT$ 
induces a $\Z$-grading on both $\kk[x,y]$ and $\Der(\kk[x,y])$. Letting $l_T$ be the  integral linear form  on $\Z^2$ associated with $T$ and 
$l_T=k$ be a supporting affine line for the Newton polytope $N(\p)$,
consider the corresponding $T$-principal part of $\partial$, 
\begin{equation}\label{eq: principal-part} \partial_T=\sum_{e\in \{l_T=k\}\cap\Z^2} \partial_{e}\in\Der(\kk[x,y])\,,\end{equation} where $\p_e$ 
stands for the homogeneous component of $\p$ of degree $e$ in decomposition \eqref{eq:decomp}. 
According to Lemma~\ref{lem:non-homog-deriv}, if $\partial$ is locally nilpotent then also $\partial_T$ is. 
In particular, for any  vertex $e$ of the Newton polytope $N(\p)$ the corresponding derivation $\p_{e}$ 
is locally nilpotent. Consequently, all the vertices of the Newton polytope $N(\p)$ are situated on the 
Demazure facets \eqref{eq:facets}. 
Hence the Newton polytope $N(\p)$ is either a quadrilateral, a triangle, a line segment, or finally a point. \end{exa}

\subsection{Degeneration techniques} \label{ss:degeneration}
We explore the $M$-grading on $\Der(A)$ in the following degeneration trick. 

\begin{prop}\label{prop: principal-part} Consider a subgroup $G\subset \Aut(X)$ 
normalized by a one-parameter subgroup $T$ of the torus $\T$. Let $H=\exp(\kk\p)$ be a $\GG_a$-subgroup 
of $G$ where $\p\in {\LND}\,(A)$, and let $\p_T\in {\LND}\,(A)$ be the $T$-principal part of $\p$. 
Then $H_T=\exp(\kk\p_T)$ is a $\GG_a$-subgroup of $\overline{G}$. \end{prop}

\begin{proof} Let $N(\p)$ be the Newton polytope of $\p$, let $l_T$  be the linear form  on $M$ associated with $T$, and let 
$$l_{\max}=\max\{l_T|_{N(\p)}\}\quad\mbox{and}\quad l_{\min}=\min\{l_T|_{N(\p)}\}\,.$$ Then one has $\p_T=\p_\tau$ 
for the face  $\tau$ of $N(\p)$  on which $l_T$ achieves its maximal value $l_{\max}$. 

The action of $T\cong\GG_m$ on $\Der(A)$ defines a $\ZZ$-grading such that any $\p\in\Der(A)$ admits a decomposition 
$$\p=\sum_{s=l_{\min}}^{l_{\max}} \p_s\quad\mbox{where}\quad 
\p_s=\sum_{e\in N(\p)\cap \{l_T=s\}} \p_{e}\in\Der(A)\quad\mbox{with}\quad  \p_{e}\in{\Der}_e(A)\,.$$
Given an isomorphism $T\cong\GG_m$, an element $t_\lambda\in T$ with $\lambda\in \GG_m$ 
acts on $A$ via $$t_\lambda . \chi^m=\lambda^{l_T(m)}\chi^m\quad\forall m\in\sigma^\vee\,.$$ 
It follows that
$$t_\lambda^{-1}\circ\p_s\circ t_\lambda=\lambda^{-s}\p_s\,.$$
Therefore, one has
$$t_\lambda^{-1}\circ \exp(\tau\p_s)\circ t_\lambda=\exp(\tau\lambda^{-s}\p_s)\,$$
and, furthermore,
$$t_\lambda^{-1}\circ \exp(\tau\p)\circ t_\lambda=\exp\left(\sum_{s=l_{\min}}^{l_{\max}} \tau\lambda^{-s}\p_s\right)\,.$$
Letting $\tau=h\lambda^{l_{\max}}$ one obtains:
$$t_\lambda^{-1}\circ \exp(\tau\p)\circ t_\lambda=\exp\left(h\sum_{s=l_{\min}}^{l_{\max}}\lambda^{(l_{\max}-s)}\p_s\right)
\longrightarrow\exp(h\p_T)\quad\mbox{as}\quad \lambda\to 0\,$$ 
on any monomial $\chi^m\in A$, $m\in\sigma^\vee$. This convergence  guarantees 
the convergence with respect to the ind-group structure on $\Aut(X)$ 
associated to any given filtration $A=\bigcup_{r=1}^{\infty} A_r$ by  finite dimensional graded subspaces of $A$ 
such that $A_r\subset A_{r+1}$, cf.~\ref{sit:3.1}.
Since $t_\lambda^{-1}\circ\exp(\tau\p)\circ t_\lambda\in G$ for any $\lambda\in\kk^*$ and $\tau\in\kk$ 
one concludes by Lemma~\ref{lem:orbit-closures}(b) that $\exp(h\p_T)\in\overline{G}$ for any $h\in\kk$. 
\end{proof}

\begin{cor}\label{cor:limit-homog-deriv} 
Under the assumptions of Proposition {\rm~\ref{prop: principal-part}} suppose that $G$ 
is normalized by the torus $\TT$. Then any vertex $e$ of the Newton polytope  $N(\p)$ belongs to 
a Demazure facet $\mathcal{S}_i$, and the root subgroup $H_e$ is contained in $\overline{G}$. 
\end{cor}

\begin{proof}  It suffices to apply Propositions~\ref{prop:homog-deriv}(c) and~\ref{prop: principal-part} to a one-parameter subgroup 
$T\subset\TT$ such that $l_T|_{N(\p)}$ achieves its maximum at $e$ and only at $e$. 
\end{proof}

\begin{lem}\label{lem:two-roots}  Letting $n\ge 2$ consider two roots $e_i\in\mathcal{S}_i\cap M$, $i=1,2$. 
Let $\delta=\langle \rho_1,e_2\rangle + 1$. Suppose that $e_2+\delta e_1\in\mathcal{S}_1$, that is, 
$\langle \rho_2,e_1\rangle \ge 1$.
Then one has 
$H_{e_2+\delta e_1}\subset\overline{\langle H_{e_1},\,H_{e_2}\rangle} \,. $
\end{lem}

\begin{proof} This follows by Corollaries~\ref{cor:NP-Ad} and~\ref{cor:limit-homog-deriv} applied to $U=\p_{\rho_1,e_1}$, $V=\p_{\rho_2,e_2}$, and
$\p=\exp(\ad_U)(V)={\Ad}_{\exp(U)}(V)$, see \eqref{eq:BCH}.
\end{proof}

\subsection{Cox ring and total coordinates}\label{ss:Cox} 
Let us recall some generalities on the Cox ring $R(X)$ of a  toric affine  variety $X$, see, e.g., 
\cite[Ch.\ 2]{ADHL}, \cite[Sect.\ 1]{Cox}, \cite[Ch.\ 5]{CLS} for detailed expositions. 

\begin{sit}\label{sit:cox-ring} As before, $\Xi=\{\rho_1,\ldots,\rho_k\}$ stands for the set of primitive ray generators of the cone $\sigma\subset N_\QQ$.
To any ray $\rho_i\in\Xi$ there corresponds a facet $\rho_i^\vee$ of the dual cone 
$\sigma^\vee$ and a $\TT$-invariant prime Weil divisor $D_i=D(\rho_i)$ on $X$. The classes $[D_1],\ldots,[D_k]$ generate the class group $\Cl(X)$.
The Cox ring $R(X)$ is the polynomial ring $\kk[x_1,\ldots,x_k]$ graded by the class group $\Cl(X)$ 
in such a way that any variable $x_i$ is a homogeneous element of degree $\deg(x_i)=[D_i]\in\Cl(X)$, $i=1,\ldots,k$. This defines the grading uniquely.  

Let $\TT(k)\cong (\GG_m)^k$ be the standard $k$-torus acting on $\A^k$, and let $F_{\Cox}={\Hom}\,(\Cl(X),\GG_m)$ 
be the dual group of the group $\Cl(X)$. Then $F_{\Cox}$ is a quasitorus, that is, the direct product of an algebraic torus and a finite Abelian group. 
By duality, $\Cl(X)$ is the group of characters of $F_{\Cox}$. The $\Cl(X)$-grading on $\kk[x_1,\ldots,x_k]$ defines an action on $\A^k$ 
of the quasitorus $F_{\Cox}\subset\TT(k)$. The structure ring $\mathcal{O}_X(X)$ is canonically isomorphic to the ring of invariants 
$\kk[x_1,\ldots,x_k]^{F_{\Cox}}$. 
This yields (canonical) isomorphisms $X\cong \A^k//F_{\Cox}$ and $\TT\cong \TT(k)/F_{\Cox}$. 

 The linear forms $\rho_1,\ldots,\rho_k$ on $M_\QQ$ define a monomorphism of  lattices $\varphi\colon M\hookrightarrow \ZZ^k$ 
 which extends to the linear embedding 
 $$\Phi\colon M_{\QQ}\hookrightarrow \A^k_\QQ,\quad v\mapsto (\langle\rho_1,v\rangle,\ldots,\langle\rho_k,v\rangle)\,.$$ 
 The coordinates of the image $\Phi(m)$ are called the \emph{total coordinates} of $m\in M$. 

We let $\Delta_{\ge 0}^\vee\subset\A^k_\QQ$ be the positive octant, and let $\hat{\mathcal{S}}_i=\mathcal{S}_i(\Delta_{\ge 0}^\vee)$ 
be the $i$th Demazure facet of $\Delta_{\ge 0}^\vee$.
The image $\Phi(e)$ of a Demazure root $e\in\mathcal{S}_i$ is a Demazure root, say, $\hat e\in\hat{\mathcal{S}}_i$. 
Any root vector $\hat e\in \Phi(M_\QQ)\cap\ZZ^k$ appears in this way. 
 The action of the root subgroup $H_e$ on $X$ induces the action of the root subgroup $H_{\Phi(e)}$ on $\A^k$, see \cite[Sect.\ 4]{Cox}. 
 In more detail, one has the following Lemma~\ref{lem:root-sbgrps-correspondence} (cf.\ \cite[Lem.\ 4.4]{Cox}). Let us introduce the necessary notation.
 \end{sit}
 
 For a lattice vector $e=(c_1,\ldots,c_k)\in\ZZ^k$ we let $x^e=x_1^{c_1}\cdots x_k^{c_k}$. 
 For $e\in\mathcal{S}_1\cap M$ one has $\hat e=(-1,c_2,\ldots,c_k)\in\ZZ^k$ where $c_i\in\ZZ_{\ge 0}$, $i=2,\ldots,k$.
The root subgroup $H_{\hat e}$ acts on $\A^k$ via
\begin{equation}\label{eq:act} (x_1,\ldots,x_k)\mapsto (x_1+tx^{\hat e+\varepsilon_1},\,x_2,\ldots,x_k),\quad t\in\kk\,,\end{equation}
where $(\varepsilon_1,\ldots,\varepsilon_k)$ is the standard basis in $\A^k$ and 
$x^{\hat e+\varepsilon_1}=x_2^{c_2}\cdots x_k^{c_k}$.

\begin{lem}\label{lem:root-sbgrps-correspondence}
\begin{itemize}
\item[{\rm (a)}] A Demazure root $\hat e\in\hat{\mathcal{S}}_i\cap\ZZ^k$ belongs to the image $\Phi(\mathcal{S}_i\cap M)$ if and only if $\deg (x^{\hat  e})=0$, 
that is, $x^{\hat e}\in  {\rm Frac}\,(\kk[x_1,\ldots,x_k])^{F_{\Cox}}$. 
\item[{\rm (b)}] The subgroups $H_{\hat e}$ and $F_{\Cox}$ of $\Aut(\A^k)$ commute  if and only if $\hat  e=\Phi(e)$ for a root $e$ of $\sigma^\vee$.
 In the latter case 
the action of the root subgroup $H_{\hat  e}$ on $\A^k$ descends to the action of the root subgroup $H_e$ on $X$ 
under the quotient morphism $\A^k\to X=\A^k// F_{\Cox}$.
\end{itemize}
\end{lem}

\begin{proof} (a) Recall that the lattice of $\TT$-invariant divisors on $X$ is generated by $D_1,\ldots,D_k$. One may assume $i=1$. One has
$$\deg(\hat e)=0\Leftrightarrow [D_1]=c_2[D_2]+\ldots+c_k[D_k]\quad\mbox{in}\,\,\,\Cl(X)\,.$$ 
The latter equality amounts to 
\begin{equation}\label{eq:div} c_2D_2+\ldots+c_kD_k-D_1={\rm div}\, (\chi^m)\in {\rm PDiv}(X)^\TT\quad\mbox{for some}\quad m\in M \,\end{equation}
where $${\rm div}\,(\chi^m)=\sum_{i=1}^k \langle \rho_i,m\rangle D_i\,.$$ Thus, \eqref{eq:div} admits a solution 
$m\in M$ if and only if $$\langle \rho_1,m\rangle=-1\quad\mbox{and}\quad
\langle \rho_i,m\rangle =c_i\ge 0\quad\forall i=2,\ldots,k\,,$$ that is, if  $m=e\in\mathcal{S}_1\cap M$ is a Demazure root and $\hat e=\Phi(e)$. 

(b) 
The  action \eqref{eq:act} on $\A^k$ commutes with the $F_{\Cox}$-action on $\A^k$ if and only if the LND 
$\p_{\varepsilon_1^\vee,\hat e}=x_2^{c_2}\cdots x_k^{c_k}\p/\p x_1$ has zero $F_{\Cox}$-degree, 
that is, if $\deg(x_1)=\deg(x^{\hat e+\varepsilon_1})$ or, which is equivalent, 
$$[D_1]=c_2[D_2]+\ldots+c_k[D_k]\,.$$
So, the first assertion of (b) follows by the argument used in the proof of (a). 
The second one is a simple consequence of the first. Indeed, the  LND $\p_{\varepsilon_1^\vee,\hat e}\in{\LND}\,(\kk[x_1,\ldots,x_k])$ 
of $F_{\Cox}$-degree zero restricts to the ring of invariants $\kk[x_1,\ldots,x_k]^{F_{\Cox}}=\mathcal{O}_X(X)$ 
yielding $\p_{\rho_1,e}\in {\LND}\,(\mathcal{O}_X(X))$. Hence the $H_{\hat  e}$-action on $\A^k$ descends to the $H_e$-action on $X$. 
\end{proof}

\begin{rem}\label{rem:normalized} The connected group $H_{\hat  e}$ normalizes $F_{\Cox}$ in $\Aut(\A^k)$ if and only if these groups 
commute. Indeed, $\Aut(F_{\Cox})$ is a finite extension of ${\rm GL}(k-n,\ZZ)$, hence a discrete group. 
\end{rem}

\section{Infinite transitivity on toric  varieties}\label{sec:toric-inf-trans}
In this section we apply Theorem~\ref{th-inf-tr} to a toric affine  variety $X$ with no torus factor. 
It is known (\cite[Thm.\ 2.1]{AKZ}) that  the action of $\SAut(X)$ on the smooth locus $\reg(X)$ is  infinitely transitive. 
Notice that the group $\SAut(X)$ is very large. Under a mild additional assumption we construct in Theorem~\ref{thm:toric-inf-trans} 
a subgroup $G\subset\SAut(X)$ which still acts infinitely transitively on~$\reg(X)$ and is generated by 
a finite number of root subgroups, as it is predicted by Conjecture~\ref{sit:conjecture}. We start with the case where $X$ is an affine space viewed as a toric variety.

\subsection{Infinite transitivity on the affine spaces: an example}\label{ss:An}
\begin{sit}\label{sit:nota-aff-space}
The affine space $\A^n=\Spec \kk[x_1,\ldots,x_n]$ can be regarded as a toric variety.  
The mutually dual lattices $N$ and $M$ are the standard lattices  of integer vectors $N=\ZZ^n\subset \A^n_\QQ$ and $M=\ZZ^n\subset(\A^n_\QQ)^*$. The cones $\sigma\subset \A^n_\QQ$ and $\sigma^\vee\subset(\A^n_\QQ)^*$ are the positive octants. The ray generators $\rho_1,\ldots,\rho_n\in N$ form the standard basis of $\A^n_\QQ$. The  dual basis $(\varepsilon_1,\ldots,\varepsilon_n)$ is the  standard  base of the lattice $M$.  The LNDs associated with the Demazure roots $e_i=-\varepsilon_i\in\mathcal{S}_i$ are 
the partial derivatives $$\p_i=\p/\p x_i=\p_{\rho_i,e_i}\in {\LND}\,(A),\quad  i=1,\ldots,n\,.$$ For a lattice vector $m=
(m_1,\ldots,m_n)\in M$ we write $x^m=x_1^{m_1}x_2^{m_2}\cdots x_n^{m_n}$. 
Given a root vector $e\in \mathcal{S}_i\cap M$ the associated root subgroup 
\begin{equation}\label{H-e} H_e=\exp(\kk x^{e+\varepsilon_i}\p_i)\subset\SAut(\A^n)\,\end{equation}
 acts on $\A^n$ via elementary transformations 
 $$x=(x_1,\ldots,x_n)\mapsto (x_1,\ldots,x_{i-1}, x_i+tx^{e+\varepsilon_i},x_{i+1},\ldots,x_n)\quad\mbox{where}\quad t\in\kk\,.$$ 
 For instance,  letting $H_{i,j}=\exp(\kk x_j^2\p_i)$ where $j\neq i$ the root subgroups  $H_{1,2}$ and $H_{2,3}$ act on $\A^n$ via 
\begin{equation}\label{eq:H-1-2} (x_1,\ldots,x_n)\mapsto (x_1+tx_2^2,x_2,\ldots,x_n),\,\, \mbox{resp.,}\,\, (x_1,\ldots,x_n)\mapsto (x_1,x_2+tx_3^2,x_3,\ldots,x_n),\,\,t\in\kk\,.
\end{equation}
To simplify the notation we will write just the coordinates of the image for such an action. 
The following result confirms Conjecture~\ref{sit:conjecture}  in the 
case $X=\A^n$, $n\ge 2$.
\end{sit}

\begin{thm}\label{thm:IV} Consider the natural action of the symmetric group $\mathbb{S}(n)$ on $\A^n$ by permutations. 
Then for any $n\ge 3$  the subgroup $$G=\langle H_{1,2},\,\mathbb{S}(n)\rangle \subset\Aut(\A^n)\,$$
acts infinitely transitively on $\mathscr{O}_G=\A^n\setminus\{0\}$. 
\end{thm}

The following corollary is straightforward.
\begin{cor}\label{cor:SLn}
For $n\ge 3$ the subgroup $\langle H_{1,2},{\SL}\,(n,\kk)\rangle \subset\Aut(\A^n)$ acts infinitely transitively on $\A^n\setminus\{0\}$. 
\end{cor}

The proof of Theorem~\ref{thm:IV} is preceded by the following lemmas. 

\begin{lem}\label{lem:first}
Assume $H_u\subset \overline{G}$ where $u=(-1,c_2,\ldots,c_n)\in \mathcal{S}_1\cap M$ with $c_2\ge 1$.
Letting $v=(0,-1,2,0,\ldots,0)\in \mathcal{S}_2\cap M$ consider the root vector 
$$e=u+v=(-1,\,c_2-1,\,c_3+2,\,c_4,\ldots,c_n)\in \mathcal{S}_1\cap M\,.$$ Then $H_e\subset \overline{G}$.
\end{lem}
   
\begin{proof} This follows immediately from Lemma~\ref{lem:two-roots}. Indeed,  the pair $(u,v)$ satisfies the assumptions of this lemma with $\delta=1$.
\end{proof}

\begin{lem}\label{lem:1}
 For three indices $s,i,j\in\{2,\ldots,n\}$ where $i\neq j$ consider a root vector of the form
\begin{equation}\label{eq:w} w=(-1,1,\ldots,1,2)+3k_s\varepsilon_s+k_{i,j}(\varepsilon_i+\varepsilon_j)\in \mathcal{S}_1\cap M\,.\end{equation}
Then
$H_w=\exp(\kk x^{w+\varepsilon_1}\p_1)\subset\overline{G}$
for any $k_s,\,k_{i,j}\in\ZZ_{\ge 0}$.
   \end{lem}
   
   \begin{proof} Let $v_i=-\varepsilon_i+2\varepsilon_{i+1}$, $i=1,\ldots,n-1$. 
The Demazure root $u=v_1=(-1,2,0,\ldots,0)\in\mathcal{S}_1\cap M$ generates the root subgroup $H_u=H_{1,2}\subset G$. 
Starting with $u$ and adding $v_2,\ldots,v_{n-1}$ one gets the root vector $w_0=(-1,1,\ldots,1,2)\in\mathcal{S}_1\cap M$. 
By Lemma~\ref{lem:first} the associated root subgroup $H_{w_0}$ is contained in $\overline{G}$. The same conclusion holds if one adds to $w_0$ the lattice vectors 
$$(-\varepsilon_i+2\varepsilon_j)+(2\varepsilon_i-\varepsilon_j)=\varepsilon_i+\varepsilon_j\quad\mbox{and}\quad (\varepsilon_i+\varepsilon_j)+(2\varepsilon_i-\varepsilon_j)=3\varepsilon_i\,.$$ Iterating one arrives at the desired conclusion. 
\end{proof}

We need also the following elementary lemma.
 
\begin{lem}\label{lem:2} For $n\ge 4$  the  vectors $$3\varepsilon_i\quad\mbox{and}\quad\varepsilon_i+\varepsilon_j, \quad i\neq j,\,\, i,j\ge 2$$
span the sublattice $L=\langle \varepsilon_2,\ldots,\varepsilon_n\rangle\subset M$ of rank $n-1$. 
   \end{lem}
   
      \begin{proof}
      One has $$\varepsilon_2=(\varepsilon_3+\varepsilon_4)+2(\varepsilon_2+\varepsilon_4)-(\varepsilon_2+\varepsilon_3)
-3\varepsilon_4\,.$$ Similar expressions hold for $\varepsilon_i$, $i=3,\ldots,n$. 
         \end{proof}
         
         \noindent \emph{Proof of Theorem}~\ref{thm:IV}. By our assumption one has $G\supset H_{i,j}$, $i\neq j$, $i,j\in\{1,\ldots,n\}$. 
         By Lemma~\ref{lem:1}, $\overline{G}\supset H_{w}$ for any root vector $w$ in \eqref{eq:w}. 
Letting $w_0=(-1,1,\ldots,1,2)$ consider $$\hat \p_1=x^{w_0}\p/\p x_1   
\quad\mbox{and}\quad A_1=\kk[x_i^3,x_ix_j\,|\, i\neq j,\,\, i,j\ge 2]\subset\ker\hat\p_1\,.$$ 
By Lemma~\ref{lem:1}, $\exp(\kk f\hat\p_1)\subset\overline{G}$ for any $f\in  A_1$. 
The conjugation by the $\mathbb{S}(n)$-action yields a collection $\{(\hat\p_i,A_i)\}_{i=1,\ldots,n}$ 
where $\exp(\kk f\hat\p_i)\subset \overline{G}$ for any $f\in  A_i$.

By Lemma~\ref{lem:2} for $n\ge 4$ this collection satisfies condition ($\beta$) of~\ref{sit-2.1}. For $n=3$ 
it satisfies condition ($\alpha$) of~\ref{sit-2.1}, that is,  the function  field ${\rm Frac}\,(\kk[x_1,x_2,x_3])$ 
is generated by $\{{\rm Frac}\,(A_i)\}_{i=1,2,3}$. By Theorem~\ref{th-inf-tr}, $\overline{G}$ 
acts infinitely transitively on the open orbit $\mathscr{O}_{\overline{G}}$. By virtue of 
Proposition~\ref{lem:many-points}(c) the same is true for $G$ and the open orbit $\mathscr{O}_G=\A^n\setminus\{0\}$.  \qed
        
        \begin{rems}\label{rem:n-2} 
  1.      Theorem~\ref{thm:IV} does not hold any longer if one replaces $x_1+tx_2^2$ in \eqref{eq:H-1-2} by $x_1+tx_2^k$ 
  with $k\ge 3$. Indeed, under such a replacement any $g\in G$ sends the pair of points $(Q,\,\omega Q)$ 
  with $Q\in\A^n\setminus\{0\}$ and $\omega^{k-1}=1$ to a pair $(g(Q),\,\omega g(Q))$. Thus, the 2-transitivity of $G$ fails. 

2.  Theorem~\ref{thm:IV} does not hold in the case $n=2$. More generally, fixing $a,b\in\ZZ_{\ge 0}$ consider the root subgroups 
$$H_1:\,\,(x,y)\mapsto (x+t_1y^a,y)\quad\mbox{and}\quad H_2: \,\,(x,y)\mapsto (x, y+t_2x^b),\quad t_1,t_2\in\kk\,.$$

\noindent {\bf Claim.} \emph{If the group $G=\langle H_1,\,H_2\rangle$ acts $2$-transitively on its open orbit then one has $ab=2$.
}

\smallskip

\begin{proof} Assume first $ab=0$; let, say, $a=0$. Then $H_1$ acts on $\A^2$ by translations, 
and $G$ acts on the first coordinate also by translations. Hence one has $x(g.P)- x(g.Q)=x(P)-x(Q)$ for any $P,Q\in\A^2$ and any $g\in G$. 
The latter is an obstacle to the 2-transitivity.

In the case $a=b=1$ one has $G=\SL(2,\kk)$. However, a linear group preserves the collinearity, hence it does not act 2-transitively on $\A^n$ for $n\ge 2$.

Assume further $ab>2$. Fixing a primitive root of unity $\omega$ of degree $ab-1>1$ 
consider the set 
$$S=\{(P,Q)\in \A^2\times\A^2\,|\,P=(x,y),\,Q=(\omega x,\,\omega^b y)\}\,.$$ It is easily seen that 
$S$ is invariant under the diagonal action of $G$ on $\A^2\times\A^2$. Once again, this serves as an obstacle to the 2-transitivity. 
\end{proof}

It would be interesting to determine the degree of transitivity of the $G$-action on $\A^2\setminus\{0\}$ 
in the remaining case $ab=2$. We can show that this action is $2$-transitive, but we do not know whether a higher transitivity holds. 
\footnote{However, see the \emph{Added in proof}.}
\end{rems} 

\subsection{Infinite transitivity on $\A^n$ and cotameness}\label{ss:cotame}

\begin{sit}\label{sit:aff-grp}  Let ${\Aff}_n$ stand for the group of affine transformations of the affine space $\A^n$, and let 
$${\SAff}_n=\{f\in{\Aff}_n\,|\,{\rm Jac}\,(f)=1\}=\langle  {\Transl}_n,\,{\SL}\,(n,\kk)\rangle$$ be 
the subgroup of volume preserving affine transformations, where Jac stands for the Jacobian. 
Notice that the subgroup of translations ${\Transl}_n$ is generated by the $\mathbb{G}_a$-subgroups 
$$ H_i=\exp(\kk\p_i),\,\,\,i=1,\ldots,n.$$ Letting
$$H(i,j)=\exp(\kk x_j\p_i),\quad i,j\in\{1,\ldots,n\},\,\, i\neq j\,$$ 
the group ${\SL}\,(n,\kk)$ is generated by $H(n,1)$ and the root subgroups $\{H(i,i+1)\}_{i=1,\ldots,n-1}$ 
corresponding to the root system of type $A_{n-1}$. Therefore,
\begin{equation}\label{eq:n+1} {\SAff}_n=\langle H_1,\, H(1,2),\ldots,H(n-1,n), H(n,1)\rangle\,\end{equation}
is generated by $n+1$ root subgroups. (In fact, there exists a smaller generating set.)
\end{sit}

\begin{defi}[\emph{Cotame automorphisms}; cf.\ \cite{EL}]\label{sit:tame} Let $\Tame_n$ stand for the tame subgroup of $\Aut(\A^n)$. 
Consider the subgroup $${\STame}_n=\{g\in{\Tame}_n\,|\,\Jac(g)=1\}\subset {\Tame}_n.$$ 
One says that $h\in\Aut(\A^n)\setminus  {\Aff}_n$ is \emph{cotame} if $\langle {\Aff}_n,h\rangle\supset {\Tame}_n$ 
and \emph{topologically cotame} if $\overline{\langle {\Aff}_n,h\rangle}\supset {\Tame}_n$. 
\end{defi}

The following result due to Edo (\cite[Thm.\ 1.2]{Edo}) extends and refines the earlier results of  
Bodnarchuk (\cite[Thm.\ 3]{Bod01}) and Furter (\cite[Thm.\ D]{Fur15}). 

\begin{thm}\label{thm:I} 
For $n\ge 2$ any element $h\in \Aut(\A^n)\setminus\Aff_n$ is topologically cotame. 
\end{thm}

\begin{rem}\label{rem:cotame} Recall that the triangular (de Jonqu\`eres) subgroup $\mathcal{B}_n\subset\Aut(\A^n)$ 
is the subgroup generated by the torus $\TT$ and the triangular root subgroups 
$\exp(x^m\p_i)$ where $x^m=x_{i+1}^{m_{i+1}}\cdots x_{n}^{m_{n}}$, $i=1,\ldots,n$. 
It is known (\cite{Bod02-2}, \cite[1.4 and Thm.\ 1.8]{Bod05}, \cite{EL}) that for $n\ge 3$  
any triangular  $h\in \Aut(\A^n)\setminus\Aff_n$ is cotame, while there is no triangular cotame $h\in \Aut(\A^2)\setminus\Aff_2$.
\end{rem}

Using Theorem~\ref{thm:I} it is not difficult to deduce the following result on
 infinite transitivity (see \cite[Thm.\ 1.2]{Bod00, Bod01, Bod05}).  

\begin{thm}\label{cor:I} For any $n\ge 2$  the group 
$\langle {\Aff}_n,h\rangle$ acts  infinitely transitively on $\A^n$ whatever is $h\in \Aut(\A^n)\setminus\Aff_n$.
\end{thm}

\begin{proof}
According to Theorem~\ref{thm:I},  $\Aff_n$ is a maximal proper closed subgroup in $\Aut(\A^n)$. 
Its normalizer $\mathcal{N}_n$ is a closed subgroup of $\Aut(\A^n)$ containing $\Aff_n$. 
Since  $\Aff_n$ is not a normal subgroup of $\Aut(\A^n)$, that is, $\mathcal{N}_n\neq \Aut(\A^n)$, one has  $\mathcal{N}_n=\Aff_n$. 

It follows that $h\notin\Aff_n$ does not normalize $\Aff_n$, that is,
$h\Aff_nh^{-1}\neq\Aff_n$. Pick up $g\in h\Aff_nh^{-1}\setminus\Aff_n$. By Theorem~\ref{thm:I} one has
$\overline{\langle {\Aff}_n,g\rangle}\supset {\Tame}_n$.
Letting $$G=\langle {\Aff}_n,h{\Aff}_n h^{-1}\rangle\supset \langle {\Aff}_n,g\rangle$$ one obtains $\overline{G}\supset {\Tame}_n$.
Since ${\Tame}_n$ acts infinitely transitively on $\A^n$ then also $\overline{G}$ does. 
The group $G$ is algebraically generated. By Proposition~\ref{lem:many-points}(c), $G$ 
acts infinitely transitively on $\A^n$. Since $G\subset \langle {\Aff}_n,h\rangle$ the latter group does as well.
\end{proof}

The following  lemma is well known. For the sake of completeness we provide an argument. 

\begin{lem}\label{lem:S} Consider an element 
$g=\alpha_1g_1\cdots\alpha_l g_l\alpha_{l+1}\in\Aut(A^n)$
where
\[\Jac(g_i)=1,\,\,\,i=1,\ldots,l,\quad\mbox{and}\quad\alpha_j\in{\Aff}_n,\,\,\, j=1,\ldots,l+1\,.\] 
Assume $\Jac(g)=1$. Then one can find another decomposition $g=\beta_1h_1\cdots\beta_l h_l\beta_{l+1}$ where
\begin{equation}\label{eq:jac} 
\Jac(h_i)=1,\,\,\,i=1,\ldots,l\quad\mbox{and}\quad\beta_j\in{\SAff}_n,\,\,\,j=1,\ldots,l+1\,.\end{equation}
\end{lem}

\begin{proof}
Letting $d_j=\Jac(\alpha_j)\in\kk\setminus\{0\}$, $j=1,\ldots,l+1$ and using the chain rule one obtains
\[\prod_{j=1}^{l+1} d_j=\Jac(g)=1\,.\]
Let further 
\[\gamma_0={\rm id}\quad\mbox{and}\quad \gamma_j=\left(\prod_{s=1}^j d_s\right)^{1/n}\cdot {\rm id},\,\,\,j=1,\ldots,l+1\,.\] 
Then $\Jac(\gamma_j)=\prod_{s=1}^j d_s$. In particular, $\Jac(\gamma_{l+1})=1$, and so, one can choose $\gamma_{l+1}={\rm id}$. 
It is easily seen that with this choice the elements
\[h_i=\gamma_ig_i\gamma_i^{-1},\,\,i=1,\ldots,l\quad\mbox{and}\quad\beta_j=\gamma_{j-1}\alpha_j\gamma_j^{-1},\,\,\,j=1,\ldots,l+1\]
verify \eqref{eq:jac}
and provide the desired decomposition.
\end{proof}

This lemma and Theorem \ref{cor:I} lead to the following corollary.

\begin{cor}\label{cor:I} 
For any $h\in \SAut(\A^n)\setminus\SAff_n$ one has 
$$\overline{\langle {\SAff}_n,h\rangle}\supset {\STame}_n\,.$$
\end{cor}

The group $\STame_n$ acts infinitely transitively on $\A^n$. This leads to the following result.

\begin{cor}\label{thm:II} For any $n\ge 2$ and any $h\in \SAut(\A^n)\setminus\SAff_n$ the subgroup
$\langle \SAff_n,h\rangle\subset \SAut(\A^n)$ acts infinitely transitively on $\A^n$. 
\end{cor}

Let us provide an alternative direct proof of a similar result which does not rely on the notion of cotameness.

\begin{thm}\label{thm:V} 
For any $n\ge 2$ and any non-affine root subgroup $H_u\subset\Aut(\A^n)$  the subgroup 
$$\langle {\SAff}_n, \, H_{u}\rangle \subset {\STame}_n$$
generated by $n+2$ root subgroups of $\Aut(\A^n)$ acts infinitely transitively on $\A^n$. 
Furthermore, for $n=2$ there exists a collection of three root subgroups with the latter property.
\end{thm}

\begin{proof} 
Suppose first that $n\ge 3$ and $u=(-1,c_2,\ldots,c_n)\in {\mathcal{S}}_1\cap\ZZ^n$. 
Since $H_{u}$ is not affine one has $c_2+\ldots+c_n\ge 2$. Letting 
$e_i=-\varepsilon_i$, $i=1,\ldots,n$ and assuming $c_i \ge 1$ one can deduce from Lemma~\ref{lem:two-roots} the relations
 \begin{itemize}\item[{\rm (i)}] $H_{u+e_i}\subset\overline{\langle H_u, H_{e_i}\rangle}$;
\item[{\rm (ii)}] $H_{u+e_i-e_j}\subset\overline{\langle H_u, H_{e_i-e_j}\rangle}$ $\forall j\ge 2$, $j\neq i$;
\item[{\rm (iii)}] $H_{2u+e_i-e_1}\subset\overline{\langle H_u, H_{e_i-e_1}\rangle}$.
\end{itemize}

\smallskip

\noindent {\bf Claim.} \emph{One has $H_v\subset\overline{G}:=\overline{\langle {\SAff}_n, \, H_{u}\rangle}$ 
for any root subgroup $H_v$ with $v\in \mathcal{S}_1\cap\ZZ^n$.}

\smallskip

\noindent \emph{Proof of the Claim.} Applying (ii) and (iii) subsequently to $u$ and the vectors obtained from $u$ 
on each step one can get a root $u'\in \mathcal{S}_1\cap\ZZ^n$ whose coordinates  dominate 
the corresponding coordinates of $v$ and such that $H_{u'}\subset \overline{G}$. Using now (i) one can conclude. 
\qed

\smallskip

 Applying the cyclic permutations of coordinates one can see that the Claim holds as well for any 
 Demazure root $v$.  According to Theorem~\ref{th-inf-tr} the group $\overline{G}$ acts infinitely transitively 
 on its open orbit $\mathscr{O}_{\overline{G}}$. Being algebraically generated the group $G=\langle {\SAff}_n, \, H_{u}\rangle$ 
 does as well, see Proposition~\ref{lem:many-points}(c). Since $G\supset \Transl_n$ one has $\mathscr{O}_{G}=\A^n$. 
 This gives the first assertion for $n\ge 3$. 

Let further $n=2$. Consider  two affine and one non-affine roots
$$e_2=(0,-1)\in\mathcal{S}_2, \quad v=e_2-e_1=(1,-1)\in\mathcal{S}_2,\quad\mbox{and}\quad  u=(-1,2)\in\mathcal{S}_1\,.$$ 
We claim that the group $G=\langle H_{e_2}, \,H_{u},\,H_{v}\rangle$ acts infinitely transitively on $\A^2$. Indeed, 
by (i) one has $H_{u+e_2}\subset\overline{\langle H_{u},\, H_{e_2} \rangle}$ where $u+e_2=(-1,1)\in\mathcal{S}_1$. Since 
$${\SL}(2,\kk)=\langle H_{u+e_2},\,H_{v} \rangle\quad\mbox{and}\quad {\SAff}_2=\langle {\SL}(2,\kk),\, H_{e_2}\rangle$$ 
it follows that 
$\langle {\SAff}_2,\,H_{u}\rangle\subset\overline{G}$ for the non-affine group $H_{u}$. 

The rest of the proof proceeds likewise in the case $n\ge 3$. 
Notice first that $H_{e_1}\subset{\SAff}_2\subset\overline{G}$. Let $u_i=(-1,i)\in\mathcal{S}_1$. 
We know already that $H_{u_i}\subset\overline{G} $ for $u_0=e_1$, $u_1=u+e_2$, and $u_2=u$. 
Assume by induction that $H_{u_i}\subset\overline{G}$ for $i=1,\ldots,d$ where $d\ge 2$. 
Due to (iii) one has $H_{u_{2d-1}}\subset\overline{G}$ where $u_{2d-1}=2u_{d}+v$. 
Then by (i) one has $H_{u_{2d-2}}\subset\overline{G}$ where $u_{2d-2}=u_{2d-1}+e_2$. Continuing in this way, 
we can finally show that $H_{u_{d+1}}\subset\overline{G}$. By induction, $H_{e}\subset\overline{G}$ for any 
$e=u_i\in\mathcal{S}_1\cap M$. 
Since $\SL(2,\kk)\subset\overline{G}$ as well, one has
$H_{f}\subset\overline{G}$ for any $f=(j,-1)\in\mathcal{S}_2\cap M$. It follows by the Jung-van der Kulk 
Theorem that $\overline{G}=\SAut(\A^2)$. The latter group acts infinitely transitively on $\A^2$. 
Then also $G$ does in view of Proposition~\ref{lem:many-points}(c).
\end{proof}

The following result completes the picture. 

\begin{thm}\label{thm:III}
For any $n\ge 2$ one can find three $\GG_a$-subgroups $U_1,U_2,U_3\subset\SAut(\A^n)$ such that 
$$G=\langle U_1,U_2,U_3\rangle\subset\SAut(\A^n)$$
acts infinitely transitively on $\A^n$. 
\end{thm}

\begin{proof} For $n=2$ the result follows from Theorem~\ref{thm:V}. 
Letting $n\ge 3$ consider the root vector $u=(-1,2,0,\ldots,0)\in\mathcal{S}_1$. Using Lemma ~\ref{lem:two-roots} 
one can deduce the relations
\begin{equation}\label{eq:2grps}
H_{u+e_2}=H_{e_1-e_2}=\exp({\kk x})\subset \overline{\langle H_{u},\,H_{e_2}\rangle}\cap\SL(n,\kk)\,\end{equation} 
where $x\in\sl(n,\kk)$ is the  nilpotent infinitesimal generator of $H_{e_1-e_2}$. 
According to \cite{Chi} there exists a nilpotent $y\in\sl(n,\kk)$ such that $\sl(n,\kk)={\rm Lie}\,\langle x,y\rangle$. It follows that
$$\SL(n,\kk)=\langle U_x,U_y\rangle\quad\mbox{where}\quad U_x=H_{e_1-e_2}=\exp(\kk x)\quad\mbox{and}\quad U_y=\exp(\kk y)\,.$$
By virtue of (\ref{eq:2grps}) one has
$${\SAff}_n=\langle U_x,\,U_y,\,H_{e_2} \rangle\subset \overline{\langle U_y,\,H_{e_2},\,H_{u}\rangle}\,$$ 
where $H_{u}$ is not affine. Let $G=\langle U_y,\,H_{e_2},\,H_{u}\rangle$. 
By Theorem~\ref{thm:V}, the subgroup $\langle {\SAff}_n,\,H_{u}\rangle\subset\overline{G}$ 
acts infinitely transitively on $\A^n$. Hence $\overline{G}$ and also $G$ do, due to Proposition~\ref{lem:many-points}(c).
\end{proof}

\begin{rem}\label{rem:Andrist} The first preprint version of this paper contained a weaker result claiming, 
for any $n\ge 3$, the infinite transitivity on $\A^n$ of a subgroup generated by four $\GG_a$-subgroups. 
Rafael B.~Andrist (\cite[Cor.\ 11]{An}) independently discovered three such subgroups given by explicit locally nilpotent derivations. 
\end{rem}

\subsection{Infinite transitivity on toric varieties}\label{ss:toric-inf-trans} 
\begin{sit}\label{sit:regular}
The toric varieties in this paper are supposed to be normal, in particular, smooth in codimension 1 and, moreover, smooth in codimension 2. 
For a toric  affine  variety $X$ the latter condition is equivalent to the following one: 
any two-dimensional face $\tau$ of the cone $\sigma\subset N_\QQ$ of $X$ 
is regular, that is, the pair of ray generators  $(\rho_i,\,\rho_j)$  of $\tau$  can be included in a base of the lattice $N$. 
For instance, the cone $\sigma\subset\A^3_\QQ$ with the primitive ray generators
$$\rho_i=\varepsilon_i,\,\,\,i=1,2,3,\,\,\,\rho_4=\varepsilon_1+\varepsilon_2-\varepsilon_3\,$$ 
defines a toric  threefold $X$ with a single singular point; thus, $X$  is smooth in codimension 2.
\end{sit}

The following theorem is the main result of this section.

\begin{thm}\label{thm:toric-inf-trans}
Let $X$ be a  toric affine variety of dimension $n\ge 2$ with no torus factor. 
Suppose that $X$ is smooth in codimension $2$. Then one can find a finite collection of root subgroups 
$H_1,\ldots,H_N$ such that the group $$G=\langle H_1,\ldots,H_N\rangle$$ acts infinitely transitively on the regular locus $\reg(X)$.
\end{thm}

\begin{proof}
If  $n=2$ then $X$ is smooth, hence $X\cong\A^2$. In this case the result (with $N=3$) follows from Theorem~\ref{thm:V}. 

Assume in the sequel $n\ge 3$.
Due to \cite[Thm.\ 2.1]{AKZ}, $\reg(X)$ coincides with the open orbit of  the group $\SAut(X)$. 
Inspecting the proof of \cite[Thm.\ 2.1]{AKZ} one can conclude that the same holds for the subgroup $R\subset\SAut(X)$ generated by the root subgroups.
By \cite[Prop.\ 1.5]{AFKKZ}
there exists a finite collection of root subgroups
$H_1,\ldots,H_r$ such that the group generated by $H_1,\ldots,H_r$ acts transitively on $\reg(X)$ too. 
To get infinite transitivity we need to enlarge this collection. 

Recall that $\Xi$ stands for the set of the primitive ray generators $\rho_1,\ldots,\rho_k$ of the cone $\sigma\subset N_{\QQ}$ associated with $X$.
Given a ray generator, say, $\rho_1\in\Xi$ there exists $\bar m_1\in M$ such that the hyperplane 
$L_{\bar m_1}=\{v\in N_\QQ\,|\,\langle v,\bar m_1\rangle=0\}$ is strictly supporting  for the ray 
$\QQ_{\ge 0}\rho_1$ of $\sigma$, that is, $$\langle \rho_1,\bar m_1\rangle=0\quad\mbox{and}\quad\langle \rho_j,\bar m_1\rangle>0\quad\forall j=2,\ldots,k\,.$$
Since $\rho_1$ is a primitive lattice vector its coordinates are coprime. 
So, $\langle \rho_1,\breve m_1\rangle=-1$ for some $\breve m_1\in M$. Fix $r\gg 1$ and a root vector 
$$e_1=r\bar m_1+ \breve m_1\in \mathcal{S}_1\quad\mbox{where}\quad \langle \rho_j,e_1\rangle \ge 2 \quad\quad\forall j=2,\ldots,k\,.$$
Up to renumbering one may suppose that the 2-cones $\tau_{1,2}$ and $\tau_{1,3}$ 
spanned by the pairs of ray generators $(\rho_1,\rho_2)$ and $(\rho_1,\rho_3)$, respectively, 
are two-dimensional faces of $\sigma$ containing the common ray $\QQ_{\ge 0}\rho_1$. 
By our assumption, $\tau_{1,2}$ is regular. Hence one can find $\breve m_{1,2}\in M$ 
such that $$\langle \rho_1,\breve m_{1,2}\rangle=0\quad\mbox{and}\quad  \langle \rho_2,\breve  m_{1,2}\rangle=-1\,.$$
Choose a strictly supporting hyperplane  $L_{1,2}=\{\langle v,\bar m_{1,2}\rangle =0\}$ of  
the face $\tau_{1,2}$ of $\sigma$ where $\bar m_{1,2}\in M$ satisfies
$$\langle \rho_1,\bar m_{1,2}\rangle =\langle \rho_2,\bar m_{1,2}\rangle=0\quad\mbox{and}\quad  \langle \rho_j,\bar m_{1,2}\rangle>0\quad\forall j\ge 3\,.$$
Fixing $r\gg 1$ consider the root
\begin{equation}\label{eq:e2r} e_2=r\bar m_{1,2}+ \breve m_{1,2}\in \mathcal{S}_2\,\,\,\mbox{with}\,\,\, \langle \rho_1, e_2\rangle =0,\quad \langle \rho_2, e_2\rangle =-1,\quad\mbox{and}\quad \langle \rho_j,e_2\rangle \ge 2 \quad\forall j\ge 3\,.\end{equation}
Choose a root $e_3\in \mathcal{S}_3$ in a similar fashion. Then
in the total coordinates one has
$$\hat e_1 =(-1,*,\ldots,*),\quad\hat e_2=(0,-1,*,\ldots,*),\quad\mbox{and}\quad \hat e_3=(0,*,-1,*\ldots,*)$$
where the stars are integers $\ge 2$. 

Let $\tau_{1,2}^\vee$ be the face  of $\sigma^\vee$ of codimension 2 dual to $\tau_{1,2}$, that is,
$$\tau_{1,2}^\vee=\left\{\bar m_{1,2}\in M\,|\,\,\,\begin{cases} \langle \rho_i,\,\bar m_{1,2}\rangle=0,\,\,\,&i=1,2\\ 
\langle \rho_i,\,\bar m_{1,2}\rangle\ge 0,\,\,\,&i=3,\ldots,k\,\end{cases}\right\}\,.$$
Choosing $n-2$ linearly independent primitive ray generators $\{\eta_1,\ldots,\eta_{n-2}\}$ of $\tau_{1,2}^\vee$ consider the sequence of roots 
\begin{equation}\label{eq:u-i} u_1=e_2,\,u_2=e_2+\eta_1,\ldots,u_{n-1}=e_2+\eta_{n-2}\in \mathcal{S}_2\cap M\,\end{equation}
 with total coordinates $\hat u_i=(0,-1,*,\ldots,*)$ where $"*"\ge 2$.
The lattice vectors 
\begin{equation}\label{eq:v-i} v_1=u_1+e_3=e_2+e_3,\,\,\,v_2=u_2+e_3,\,\,\ldots,\,\,v_{n-1}=u_{n-1}+e_3\in\tau_1\,\end{equation}
have total coordinates $\hat v_i=(0,*,\ldots,*)$ where  $"*"\ge 1$.
We claim that these vectors are linearly independent, that is, $v_1=e_2+e_3\notin\span(\eta_1,\ldots,\eta_{n-2})=:W$.  
Indeed, any $\eta\in W$ has total coordinates $(0,0,*,\ldots,*)$ whereas $\langle\rho_2,v_1\rangle>0$. 

\smallskip

\noindent {\bf Claim.} \emph{Consider the cone $\omega\subset \tau_1$ of dimension $n-1$ with ray generators 
$v_1,\ldots,v_{n-1}$. Consider also the submonoid $\mathcal{M}_1=\ZZ_{\ge 0}v_1+\ldots+\ZZ_{\ge 0}v_{n-1}$ of $\omega$ of rank $n-1$ and the subgroup
$$G_1=\langle H_{e_1}, H_{u_1},  H_{u_2},\ldots,H_{u_{n-1}},H_{e_3}\rangle\subset\SAut(X)\,.$$ 
Then one has $H_w\subset\overline{G}_1$ for any root $w\in e_1+\mathcal{M}_1\subset\mathcal{S}_1\cap M$.
}

\smallskip

\noindent \emph{Proof of the claim}. 
The assertion is true for $w=e_1$.
Assume by recursion that  $H_w\subset\overline{G}_1$ for some root $w\in e_1+\mathcal{M}_1$. 
It suffices to show that then the same holds as well for any root $w+v_i\in e_1+\mathcal{M}_1$, $i=1,\ldots,n-1$.

Notice that $\hat w=(-1,*,\ldots,*)$ where  $"*"\ge 2$. Since $\langle\rho_2,w\rangle\ge 1$
the pair $(w,u_i)$ satisfies the assumptions of Lemma~\ref{lem:two-roots} with $\langle\rho_1,u_i\rangle=0$ and $\delta=1$
for any $i=1,\ldots,n-1$. Applying the recursive hypothesis and Lemma~\ref{lem:two-roots} one deduces that 
$$H_{w+u_i}\subset\overline{\langle H_w, H_{u_i}\rangle}\subset\overline{G}_1\quad\forall i=1,\ldots,n-1\,.$$
Likewise, since $\langle\rho_3,w+u_i\rangle\ge 1$
the pair $(w+u_i,e_3)$ satisfies the assumptions of Lemma~\ref{lem:two-roots} with $\langle\rho_1,e_3\rangle=0$ and $\delta=1$
for any $i=1,\ldots,n-1$.
Applying Lemma~\ref{lem:two-roots} one deduces by \eqref{eq:v-i} that
$$H_{w+v_i}\subset\overline{\langle  H_{w+u_i}, H_{e_3} \rangle}\subset\overline{G}_1\quad\forall i=1,\ldots,n-1\,.$$
 This yields the inductive step and ends the recursion. \qed

\smallskip

Now one can constitute the data verifying the assumptions~\ref{sit-2.1} of Theorem~\ref{th-inf-tr}. 
Recall that we fixed already   a collection of root subgroups $H_1,\ldots,H_r$ such that the open orbit of the group 
$\langle H_1,\ldots,H_r\rangle$ coincides with $\reg(X)$. 

Letting $\p_1=\p_{\rho_1,e_1}\in\LND(\mathcal{O}_X(X))$ consider the subalgebra 
$$A_1=\kk[\chi^v\,|\,v\in  \mathcal{M}_1]=\kk[\chi^{v_1},\ldots,\chi^{v_{n-1}}]\subset \ker(\p_1)\,,$$ see Remark~\ref{rem:4.5}.
According to the Claim for any $f\in A_1$ the replica $\exp(\kk f\p_1)$
of $H_{e_1}$ is a subgroup of $\overline{G}_1$. Since ${\rm rank}\,(\mathcal{M}_1)=n-1$ one has 
$$[{\rm Frac}(\ker(\p_1)):{\rm Frac}(A_1)]<+\infty\,.$$
 Hence there exists $b_1\in \ker\partial_1$ such that ${\rm Frac}\,(\ker\partial_1)$ is generated by $b_1$ 
 and ${\rm Frac}\,(A_1)$.  According to Remark~\ref{rem:4.5} one can write $b_1=\sum_{j=1}^s c_j\chi^{m_j}$ 
 where $m_j\in \tau_1\cap M$. Then $H=\exp(\kk b_1\partial_1)$ is contained in the product 
 of the commuting root subgroups $H_{r+j}:=\exp(\kk\chi^{m_j}\p_1)$, $j=1,\ldots,s$.

Choose linearly independent ray generators $\rho_1,\ldots,\rho_n\in\Xi$.
Repeating the same construction one obtains for any $i=1,2,\ldots,n$ a triple $(G_i, \p_i, A_i)$ 
with properties similar to the ones of  $(G_1, \p_1, A_1)$. Let now 
$$G=\langle H_1,\ldots,H_{r+s},G_1,\ldots,G_n\rangle\subset\SAut(X)\,.$$
The group $\overline{G}$ satisfies the assumptions~\ref{sit-2.1}($\gamma$) of Theorem~\ref{th-inf-tr}. 
Due to this theorem, $\overline{G}$ acts infinitely transitively on its open orbit 
$\mathscr{O}_{\overline{G}}=\mathscr{O}_G=\reg(X)$. By virtue of Proposition~\ref{lem:many-points}(c) the same is true for $G$. 
\end{proof}

\subsubsection{Final remarks}

Theorem~\ref{thm:toric-inf-trans} leads to the following questions.

\begin{sit}\label{Q1} {\bf Problem.} \emph{Let $X$ be a toric affine variety  with no torus factor.
What is the smallest number of root subgroups {\rm (}$\GG_a$-subgroups, respectively{\rm )} $H_1,\ldots,H_s$  of $\Aut(X)$ 
such that $G=\langle H_1,\ldots,H_s \rangle$ acts on $X$ with an open orbit and is infinitely transitive on this orbit?}
\end{sit}
We cannot exclude that this number equals 2,  at least in the setup of arbitrary $\GG_a$-subgroups; cf.\ Theorem~\ref{thm:III}. 
Let us remind
the question (V.~L.~Popov \cite[Problem 3.1]{Pop05}) as to when (the closure of) the subgroup $G=\langle H_1, H_2\rangle$ 
generated by $\GG_a$-subgroups $H_1, H_2\subset\Aut(\A^n)$ is an algebraic group. The third author thanks 
Hanspeter Kraft for an inspiring example of two root subgroups of $\Aut(\A^2)$ whose product is an (infinite dimensional) free product.
The discussions with  Hanspeter Kraft  resulted in the following theorem (\cite[Thm.\ 5.5.1]{KZ}) which answers, in particular, the  question above.

\noindent  \begin{thm} Given an affine variety $X$ the subgroup $G\subset\Aut(X)$ generated by a family $\mathfrak{F}$ 
of connected algebraic subgroups of $\Aut(X)$ is a {\rm (}closed{\rm )} algebraic group if and only if the Lie algebras ${\rm Lie}\,(H)$ for $H\in\mathfrak{F}$ generate a finite dimensional Lie algebra.
\end{thm}

The following conjecture arises naturally (cf., e.g., Lemma~\ref{lem:two-roots}).

\begin{conj}\label{conj:unip-sbgrps} \emph{Let $X$ be an affine variety, and 
let $A=\mathcal{O}_X(X)$ be its structure algebra. Consider the group $G=\langle H_1,\ldots,H_k\rangle$ 
generated by a finite collection of $\GG_a$-subgroups $H_i=\exp(\kk \p_i)\subset\SAut(X)$ 
where $\p_i\in\LND(A)$, $i=1,\ldots,k$. Then the $\GG_a$-subgroup $H=\exp(\kk\p)\subset\SAut(X)$ where $\p\in\LND(A)$ is contained in $\overline{G}$
 if and only if $\p\in{\rm Lie}\,\langle \p_1,\ldots,\p_k\rangle$.}
\end{conj}

Of course, the latter holds if $G$ is an algebraic group. One more justification is provided by Lemmas~\ref{lem:iterated} 
and~\ref{lem:two-roots}. Indeed, letting $X=\Spec(A)$ be a nondegenerate toric affine variety of dimension $n\ge 2$, 
in the notation of Lemma~\ref{lem:iterated} for two LNDs 
$U=\p_1$ and $V=\p_2$ of $A$ and for $m=\delta$ the nonzero homogeneous derivation $W=\ad_{U}^\delta(V)$
is an LND. According to Lemma~\ref{lem:two-roots} the associated root subgroup $H_{W}=\exp(\kk W)$ 
is contained in the closure $\overline{G}$ where $G=\langle H_U,\,H_V\rangle$.

\medskip

\noindent \emph{Added in proof}.  For $n>3$ one can equally deduce Theorem~\ref{thm:III} by using \cite{Chi} 
and \cite[Thm.\ 1.4]{KBYE}. However, this does not work for $n=2$ and $n=3$. 

Answering our question in Remark~\ref{rem:n-2}.2, D.~Lewis, K.~Perry, and A.~Straub proved recently (see \cite[Cor.\ 21]{LPS}) that for $n=2$ 
the group generated by the root subgroups 
\[H_1=\{(x,y)\mapsto (x+t_1y^2,y)\}\quad\mbox{and}\quad H_{2}=\{(x,y)\mapsto (x,y+t_2x)\},\quad t_1,\,t_2\in\kk\] 
acts infinitely transitively on $\A^2\setminus 0$. It remains unclear, however, whether for $n=2$ in Theorem~\ref{thm:III} 
one can find a subgroup $G\subset\Aut(\A^2)$ which acts infinitely transitively on $\A^2$ and is generated 
by just two one-parameter $\GG_a$-subgroups of $\Aut(\A^2)$.

\end{document}